\documentclass[12pt]{amsart}
\usepackage[a4paper, left=3.1cm, right=3.1cm, bottom=3cm,top=3cm]{geometry}
\usepackage{mathscinet}
\usepackage{amsmath, amsfonts, amssymb, amsthm}
\usepackage{xcolor}
\definecolor{blue}{rgb}{0,0.5,0.5}
\definecolor{ocean}{rgb}{0.00,0.26,0.50}
\usepackage[colorlinks = true, linkcolor=blue, citecolor=ocean]{hyperref}
\newtheorem{theorem}{Theorem}[section]
\newtheorem{proposition}[theorem]{Proposition}
\newtheorem{definition}[theorem]{Definition}
\newtheorem{remark}[theorem]{Remark}

\newtheorem{example}[theorem]{Example}

\newtheorem{corollary}[theorem]{Corollary}
\newtheorem{lemma}[theorem]{Lemma}

\def\<{\,<\!}
\def\>{\!>\,}

\begin{document}
	
\title[On Combination Theorems and Bowditch Boundaries]{On Combination Theorems and Bowditch Boundaries of Relatively Hyperbolic TDLC Groups}
	
\author{Swarnali Datta, Arunava Mandal, and Ravi Tomar}

\address{Department of Mathematics,
Indian Institute of Technology Roorkee,
Uttarakhand 247667, India}
\email{swarnali\_d1@ma.iitr.ac.in}

\address{Department of Mathematics,
Indian Institute of Technology Roorkee,
Uttarakhand 247667, India}
\email{arunava@ma.iitr.ac.in}

\address{Beijing International Center for Mathematical Research, Peking University, No. 5 Yiheyuan Road Haidian District, Beijing, P.R.China 100871}
\email{ravitomar547@gmail.com}	

	\begin{abstract}
    Based on the work of Farb, Bowditch, and Groves-Manning on discrete relatively hyperbolic groups, we introduce an approach to relative hyperbolicity for totally disconnected locally compact (TDLC) groups. For compactly generated TDLC groups, we prove that this notion is equivalent to the one introduced by Arora-Pedroza. Let $G=A\ast_C B$ or $G=A\ast_C$ where $A$ and $B$ are relatively hyperbolic TDLC groups and $C$ is compact. We prove that $G$ is a relatively hyperbolic TDLC group and give a construction of the Bowditch boundary of $G$. As a consequence, we prove that if the rough ends of $G$ are infinite, then the topology of the Bowditch boundary of $G$ is uniquely determined by the topology of the Bowditch boundary of $A$ and $B$. Further, we show that if a relatively hyperbolic TDLC group has one rough end, then its Bowditch boundary is connected. Finally, we show that if the Gromov boundary of a hyperbolic TDLC group $G$ is totally disconnected, then $G$ splits as a finite graph of compact groups.
	\end{abstract}
\maketitle
    \subjclass{\it 2020 Mathematics Subject Classification:} Primary 20F65, 20F67; Secondary 22D05 

\keywords{\bf Keywords:} Hyperbolic and relatively hyperbolic groups, TDLC groups, Graphs of topological groups, Gromov boundary, Bowditch boundary.

\setcounter{tocdepth}{1}
\tableofcontents

 \section{Introduction}
In his seminal essay, Gromov \cite{gromov-hypgps} treated infinite groups as geometric objects and introduced hyperbolic and relatively hyperbolic groups. One of the central themes of geometric group theory is to prove combination theorems for groups having negative curvature properties. Bestvina-Feighn \cite{BF} proved a celebrated combination theorem for a finite graph of hyperbolic groups. Motivated by the work of Bestvina-Feighn, several authors proved combination theorems for graphs of hyperbolic and relatively hyperbolic groups \cite{ilyakapovichcomb}, \cite{dahmani-comb}, \cite{mahan-reeves}, \cite{mahan-sardar},
\cite{martin1}, \cite{tomar}.

The article is a part of a program of generalizing geometric techniques in the study of discrete groups to a larger class of totally disconnected locally compact groups, which includes algebraic groups over non-Archimedean local fields, and automorphism groups of locally finite graphs. A locally compact group is said to be {\em totally disconnected} if the identity element is its own connected component. From now on, we abbreviate totally disconnected locally compact groups as TDLC groups. A TDLC group has been a topic of interest in the last three decades after the foundational work by G. Willis \cite{willis-tdlc}, pioneering insights into tidy subgroups and the scale function. In this article, we extend classical geometric techniques \text{---} originally developed for discrete groups \text{---} to the broader category of TDLC groups. Specifically, we prove combination theorems for relatively hyperbolic TDLC groups.  Moreover, we introduce and adapt the two different notions of a relatively hyperbolic TDLC group, proving their equivalence with the existing notion defined by Arora-Pedroza \cite{arora-pedroja}. In turn, this enables us to define the notion of the boundary of a relatively hyperbolic TDLC group. Then, we further study the topology of the boundary of graphs of relatively hyperbolic TDLC groups. This article, in our view, builds a foundation for further study of relatively hyperbolic TDLC groups. 
 
For infinite finitely generated groups, the Cayley graph is a fundamental instrument in studying their geometry. In the same spirit, to study large-scale properties of TDLC groups, Kr\"{o}n and M\"{o}ller \cite{kron-moller} introduced the notion of a  Cayley-Abels graph (with the name rough Cayley graphs but now commonly known as Cayley-Abels graph) for compactly generated TDLC groups, see Subsection \ref{subsection-cayley-abels-graphs} for the definition and further details. Most importantly, they show that any two Cayley-Abels graphs of a compactly generated TDLC group are quasiisometric. This leads to defining the notion of a hyperbolic TDLC group: a TDLC group $G$ is said to be {\em hyperbolic} if $G$ is compactly generated and some (hence any) Cayley-Abels graph of $G$ is a Gromov hyperbolic space. Let $\mathcal H$ be a finite collection of open subgroups of a TDLC group $G$. Recently, Arora-Pedroza \cite{arora-pedroja} defined a relative version of a Cayley-Abels graph for $G$ with respect to $\mathcal H$ (Definition \ref{definition-cayley-abels-graph-with-resepct-to}). Further, they introduced the notion of relative hyperbolicity in this context, adapting the notion of relative hyperbolicity developed by Osin \cite{osin-book} and Bowditch \cite{bowditch-relhyp} in the discrete set-up. We prove the following in this framework, which is well known in the discrete setting. {\em Throughout the paper, on a finite graph of topological groups, we use the topology as in Proposition \ref{proposition-topology-graphs-of-groups}}.

\begin{theorem}\label{rel-combination-theorem-amalgam}
Suppose $A$ and $B$ are TDLC groups which are compactly generated relative to a finite collection of open subgroups $\mathcal H_A$ and $\mathcal H_B$, respectively. Suppose $G=A\ast_C B$ is an amalgamated free product such that $A$ and $B$ are hyperbolic relative to $\mathcal H_A$ and $\mathcal H_B$, respectively, and $C$ is compact. Then, $G$ is hyperbolic relative to $\mathcal H_A\cup\mathcal H_B$.
\end{theorem}

For HNN extensions, we prove the following:

\begin{theorem}\label{theorem-rel-combination-theorem-hnn}
Suppose $A$ is a TDLC group which is compactly generated relative to a finite collection of open subgroups $\mathcal H$.
Suppose $G=A\ast_C$ is an HNN extension such that $A$ is hyperbolic relative to $\mathcal H$ and $C$ is compact. Then, $G$ is hyperbolic relative to $\mathcal H.$      
\end{theorem}
In particular, the above results indicate that it produces more examples of TDLC hyperbolic groups. As an application, it follows that ${\rm SL}(2,\mathbb Z_p)\times \mathbb Z$ is a hyperbolic group and it contains a subgroup isomorphic to $\mathbb Z\times\mathbb Z$. This shows that having a subgroup isomorphic to $\mathbb Z\times \mathbb Z$ is not an obstruction to hyperbolicity for TDLC groups (Remark \ref{remark-Z}). One is referred to Theorem \ref{theorem-rel-hyp-combination-general-case} for the statement of the combination theorem for a graph of relatively hyperbolic TDLC groups. Baumgarter-M\"{o}ller-Willis \cite[Corollary 29]{baumgartner-moller-willis-flatrank-one} proved that a compactly generated group which splits as a finite graph of compact groups is hyperbolic. We generalize their result and, as corollaries of the previous theorems, we obtain combination theorems for hyperbolic TDLC groups, see Corollaries \ref{corollary-combination-amalgam-hyp-tdlc} and \ref{corollary-combination-hnn-hyperbolic-tdlc}.

\begin{remark}
    To prove Theorems \ref{rel-combination-theorem-amalgam} and Theorem \ref{theorem-rel-combination-theorem-hnn}, we construct a graph using the Cayley-Abels graphs of $A$ and $B$ with respect to $\mathcal H_A$ and $\mathcal H_B$, respectively. Then, we prove that this graph is a Cayley-Abels graph of $G$ with respect to $\mathcal H_A\cup\mathcal H_B$. However, one can also use Theorems 3.1 and 3.2 from \cite{bigdely-pedroza} to construct a Cayley-Abels graph of $G$ with respect to $\mathcal H_A\cup\mathcal H_B$. It is also worth mentioning that, using the combination of graphs construction from \cite{bigdely-pedroza}, one can prove a combination theorem for a finite graph of relatively hyperbolic TDLC groups with parabolic edge groups. Since our main motive for this article is to understand the Bowditch boundaries and their homeomorphism types of amalgamated free products and HNN extensions of relatively hyperbolic TDLC groups, we did not include it in this paper.
\end{remark}

In \cite{gromov-hypgps}, Gromov introduced the notion of the boundary of a hyperbolic group, now known as the Gromov boundary. Since any two Cayley-Abels graphs of a TDLC group are quasiisometric, this leads to the notion of the Gromov boundary of a hyperbolic TDLC group. In \cite{bowditch-relhyp}, Bowditch introduced the notion of the boundary of a relatively hyperbolic group, now known as the Bowditch boundary. One of our main aims in this paper is to study the Bowditch boundaries of amalgamated free product and HNN extension of relatively hyperbolic TDLC groups. For that, we need a model space for Cayley-Abels graphs of the amalgamated free product and the HNN extension so that the Bowditch boundary of the amalgam and HNN extension gives a compactification of this model space, see Subsection \ref{section-bowditch-boundaries-rel-hyp-amalgam-hnn}. Since the relative version of the Cayley-Abels graph introduced by Arora-Pedroza \cite{arora-pedroja} is not necessarily locally finite, we cannot use it to construct the model space for the amalgam and HNN extension. Thus, we adapt different definitions of relatively hyperbolic TDLC groups. Motivated by the work of Farb \cite{farb-relhyp}, \cite{bowditch-relhyp}, and \cite{groves-manning}, we give two definitions (see Definition \ref{definition-tdrh-II} and Definition \ref{definition-tdrh-III}) of compactly generated relatively hyperbolic TDLC groups and prove their well-definedness in Propositions \ref{well-definedness-augmented-cayley-graph} and \ref{propsition-welldefined-ferb}. In Theorem \ref{theorem-rh-equivalence}, we prove the equivalence of these three definitions for compactly generated TDLC groups. As a consequence, we prove the following, which is well-known in the discrete setting \cite{osin-book}. This can be thought of as an independent interest.
\begin{theorem}\label{Theorem-rel-hyp-intersection-is-compact} Let $G$ be a compactly generated TDLC group, and $\mathcal H=\{H_1,\dots,H_n\}$ be a finite collection of compactly generated open subgroups of $G$. If $G$ is hyperbolic relative to $\mathcal H$, then we have the following:
 \begin{enumerate}
     \item[{$(1)$}] For all $i\neq j$, $H_i\cap H_j$ 
     is compact. In particular, $\bigcap_{i=1}^nH_i$ 
     is compact. 
     \item[{$(2)$}] For $g\notin H_i$, $gH_ig^{-1}\cap H_j$ is compact for all $i$ and $j$.

     \item[{$(3)$}] If $H_i\in \mathcal H$ is hyperbolic for all $i=1,\dots,n$, then $G$ is hyperbolic.
 \end{enumerate} 
\end{theorem}

In \cite{martin-swiat}, Martin-\`{S}wi\polhk{a}tkowski proved that the topology of the Gromov boundary of a free product of hyperbolic groups is uniquely determined by the topology of the Gromov boundary of each free factor. Recently, the third-named author extended this result for relatively hyperbolic groups \cite{tomar-homeo-relhyp}. For Morse boundary, the result of the same flavor is proved by Zbiden \cite{zbiden-morse}. In the realm of relatively hyperbolic TDLC groups, we prove an analogous result. It is worth mentioning that, to prove this result, we are not adapting the technique used in \cite{martin-swiat}. Rather, we use the dense amalgam technique developed by \`{S}wi\polhk{a}tkowski \cite{swiatkowski-dense-amalgam}. Now, we state our main result for relatively hyperbolic TDLC groups. For the definition of rough ends of a TDLC group, one is referred to Definition \ref{definition-rough-ends}.

\begin{theorem}\label{theorem-rel-hyp-homeo-type-amal-case} Suppose $G=A\ast_C B$ and $G'=A'\ast_{C'} B'$ are amalgamated free products of compactly generated TDLC groups such that $C$ and $C'$ are compact, and the rough ends of $G$ and $G'$ are infinite. Suppose $A,B$ are hyperbolic relative to $\mathcal H_A, \mathcal H_B$, respectively, and $A',B'$ are hyperbolic relative to $\mathcal H_A',\mathcal H_B'$, respectively. Then, if the Bowditch boundaries of $ A$ and $B$ with respect to $\mathcal H_A$ and $\mathcal H_B$ are homeomorphic to the Bowditch boundaries of $A'$ and $B'$ with respect to $\mathcal H_A'$ and $\mathcal H_B'$, respectively then the Bowditch boundary of $G$ with respect to $\mathcal H_A\cup\mathcal H_B$ is homeomorphic to the Bowditch boundary of $G'$ with respect to $\mathcal H_A'\cup\mathcal H_B'$  (by Theorem \ref{rel-combination-theorem-amalgam}, $G$ and $G'$ are hyperbolic relative to $\mathcal H_A\cup\mathcal H_B$ and $\mathcal H_A'\cup \mathcal H_B'$, respectively).
\end{theorem}

\begin{remark}\label{compact-hypothesis-reason}
    If $(\mathcal G,\mathcal Z)$ is a finite graph of group such that the edge group is infinite and the vertex groups are relatively hyperbolic, then the topology of the boundary of the fundamental group of $(\mathcal G,\mathcal Z)$ is not uniquely determined by the topology of the boundaries of the vertex groups in general, see \cite[Example 7.2]{tomar}. The same thing is also not true for graphs of hyperbolic groups in general, for example, the free group of rank $4$ and the fundamental group of a closed orientable surface of genus $2$ split as amalgamated free products with isomorphic vertex groups. However, their Gromov boundaries are not homeomorphic. Thus, when $C$ and $C'$ are non-compact, the conclusion of Theorem \ref{theorem-rel-hyp-homeo-type-amal-case} is false.
\end{remark} 

For an HNN extension, we prove the following:

\begin{theorem}\label{theorem-rel-hyp-homeo-type-hnn}
Suppose $G=A\ast_C$ and $G'=A'\ast_{C'}$ are HNN extensions of compactly generated TDLC groups such that $C$ and $C'$ are compact, and the rough ends of $G$ and $G'$ are infinite. Suppose $A$ and $A'$ are hyperbolic relative to $\mathcal H_A$ and $\mathcal H_B$, respectively. Then, if the Bowditch boundaries of $A$ and $A'$ with respect to $\mathcal H_A$ and $\mathcal H_A'$, respectively, are homeomorphic, then the Bowditch boundaries of $G$ and $G'$ with respect to $\mathcal H_A$ and $\mathcal H_A'$, respectively, are homeomorphic (by Theorem \ref{theorem-rel-combination-theorem-hnn}, $G$ and $G'$ are hyperbolic relative to $\mathcal H_A$ and $\mathcal H_A'$, respectively).
\end{theorem}

Theorem \ref{theorem-rel-hyp-homeo-type-amal-case} and Theorem \ref{theorem-rel-hyp-homeo-type-hnn} generalize \cite[Theorem 1.2]{tomar-homeo-relhyp} even in the discrete setting. In particular, we now do not need hypothesis (2) in Theorem 1.2 of \cite{tomar-homeo-relhyp}. For the statement of the theorem for general graphs of relatively hyperbolic groups, see Theorem \ref{theorem-rel-hyp-homeo-type-general}. As a consequence, we obtain similar results for hyperbolic TDLC groups, see Corollaries \ref{corollary-hyp-homeo-type-amalgam} and \ref{corollary-hyp-homeo-type-hnn}. 

It is well known that the Gromov boundary of a one ended discrete hyperbolic group is connected. In \cite[Proposition 10.1]{bowditch-relhyp}, Bowditch proved that the Bowditch boundary of a one ended discrete relatively hyperbolic group is a connected topological space. By the definition of rough ends of a compactly generated TDLC group, it follows that the Gromov boundary of a hyperbolic TDLC group having one rough end is connected. We prove an analogous result in the framework of relatively hyperbolic TDLC groups. This gives a new proof of \cite[Proposition 10.1]{bowditch-relhyp} even in the discrete setup.

\begin{theorem}\label{theorem-connectedness-bowditch-boundary} Let $G$ be a compactly generated TDLC group which is hyperbolic relative to a finite collection of compactly generated open subgroups $\mathcal H$. Suppose $G$ has one rough end, then the Bowditch boundary of $G$ is connected.
\end{theorem}

Using Theorem \ref{theorem-connectedness-bowditch-boundary}, we obtain partial converses of Theorems \ref{theorem-rel-hyp-homeo-type-amal-case} and \ref{theorem-rel-hyp-homeo-type-hnn}, see Theorems \ref{theorem-converse-homeo-type-rel-amal} and \ref{converse-homeo-type-rel-hnn}. Moreover, as an application of Theorems \ref{theorem-rel-hyp-homeo-type-amal-case} and \ref{theorem-rel-hyp-homeo-type-hnn}, using the accessibility of hyperbolic TDLC groups, we obtain the following:

\begin{proposition}\label{proposition-total-disconnection-implies-splitting} Let $G$ be a hyperbolic TDLC group such that the Gromov boundary of $G$ is totally disconnected. Then, $G$ splits as a finite graph of topological groups such that the vertex groups are compact.
    \end{proposition}

The paper is organized as follows. In Section 2, we recall essential definitions, outline a useful construction of Bass-Serre trees of graphs of topological groups, and present some observations on the Cayley-Abels graph in Lemmas \ref{lemma-countable-vertices-cayley-abels-graph} and \ref{lemma-proper-embedding-subgraph}. Section 3 establishes Theorems \ref{rel-combination-theorem-amalgam} and \ref{theorem-rel-combination-theorem-hnn}, which provide a combination theorem for the amalgamated free product and the HNN extension of relatively hyperbolic TDLC groups. In Section 4, we introduce two definitions for compactly generated relatively hyperbolic TDLC groups and prove their well-definedness in Propositions \ref{well-definedness-augmented-cayley-graph} and \ref{propsition-welldefined-ferb}. We then prove the equivalence stated in Theorem \ref{theorem-rh-equivalence} and derive Theorem \ref{Theorem-rel-hyp-intersection-is-compact}. Section 5 discusses the boundaries of the amalgamated free product and HNN extension of relatively hyperbolic TDLC groups, as detailed in Theorems \ref{theorem-rel-hyp-homeo-type-amal-case} and \ref{theorem-rel-hyp-homeo-type-hnn}. Section 6 discusses the proof of Theorem \ref{theorem-connectedness-bowditch-boundary}. In Section 7, we prove partial converses of Theorems \ref{theorem-rel-hyp-homeo-type-amal-case} and \ref{theorem-rel-hyp-homeo-type-hnn}. Finally, we prove Proposition \ref{proposition-total-disconnection-implies-splitting}, which gives a splitting of a hyperbolic TDLC group having the totally disconnected Gromov boundary.
\vspace{.3cm}


\section{Preliminaries}\label{section-preliminaries}
In this section, we establish notations and recall some basic definitions and results that are relevant to us. Throughout the paper, all graphs are assumed to be connected, and each edge has length one so that graphs are naturally geodesic metric spaces. For a graph $\Gamma$, we denote by $V(\Gamma)$ and $E(\Gamma)$ the set of vertices and edges of $\Gamma$, respectively. For a metric space $(Z,d)$, a subset $R\subset Z$, and $D\geq 0$, we denote the closed $D$-neighborhood of $R$, i.e. $\{z\in Z: d(z,r)\leq D \text{ for some } r\in R\}$ by $N_D(R)$. In this paper, all topological groups are Hausdorff.

\subsection{Cayley-Abels graphs} \label{subsection-cayley-abels-graphs} Given a finitely generated group, one can construct its Cayley graph and thus treat the group as a geometric object. For TDLC groups, an analog construction has first been taken into account by Abels \cite{abels}. These are known as {\em Cayley-Abels graphs} named after him. A less technical approach to Cayley-Abels graphs is provided in \cite{kron-moller} with the name {\em rough Cayley graph}.

\begin{definition}\label{definition-cayley-abels-graph}
A locally finite connected graph $X$ is said to be a {\em Cayley-Abels graph} of a TDLC group $G$ if $G$ acts transitively on $V(X)$ and stabilizers of vertices are compact open subgroups of $G$.
\end{definition}

 A topological group is said to be {\em compactly generated} if it is algebraically generated by a compact subset. For a TDLC group $G$, a Cayley-Abels graph exists if and only if $G$ is compactly generated \cite[Theorem 2.2]{kron-moller}.

 \vspace{.2cm}

 {\bf Existence of a Cayley-Abels graph:} Let $G$ be a compactly generated TDLC group, and $U$ be a compact open subgroup of $G$. If $K$ is a compact generating set of $G$, there exists a finite symmetric set $S$ containing the identity element of $G$ such that $K\subset SU$. Define a graph $X(K,U,S)$ whose vertex set is the set of left cosets of $U$ in $G$, and the edge set is $\{\{gU,gsU\}:g\in G \text{ and } s\in S\}$. Then, it is easy to check that $X$ is a Cayley-Abels graph for $G$ (see  \cite{kron-moller}). We write $X$ in place of $X(K,U,S)$ when $K,U,S$ are clear from the context.

 Now, we note the following lemma, which is important to us for later purposes.

\begin{lemma}\label{lemma-countable-vertices-cayley-abels-graph}
Suppose $G$ is a TDLC group. Then the set of vertices of any Cayley-Abels graph of $G$ is countable.
\end{lemma}

\begin{proof}
Let $X$ be a Cayley-Abels graph of $G$. Then $G$ is compactly generated and thus $\sigma$-compact, i.e. it is a countable union of compact subsets. Consider any open subgroup $U$ of $G$. The quotient map $\pi: G \to G/U$ is open, making each coset in $G/U$ open; hence,  $G/U$ is discrete. Since $\pi$ is continuous, surjective and $G$ is $\sigma$-compact, $G/U$ is $\sigma$-compact. As $G/U$ is both discrete and $\sigma$-compact, it is countable. This shows that, for any open subgroup $U$ of $G$, the coset space $G/U$ is countable. 

For each $v\in V(X)$, the $G$-stabilizer $G_v$ of $v$ is a compact open subgroup of $G$. By the discussion in the previous paragraph, the coset space $G/G_v$ is countable. Since the action of $G$ on $V(X)$ is transitive, for any $w \in V(X)$, there exists $g_w \in G$ such that $g_w \cdot v = w$. Suppose there exists $g_w' \in G$ such that $g_w'\cdot v = w$. Then $g_w^{-1} g_w'\cdot v=v$, implying $g_w^{-1} g_w'\in G_v$, or equivalently, $g_w G_v = g_w' G_v$. Therefore, we define a map $$\varphi: V(X) \to G/G_v\;\;{\text by}\;\;\varphi(w) = g_w G_v.$$ This map is injective: if $\varphi(w_1)=\varphi(w_2)$, then $g_{w_1}G_v=g_{w_2}G_v$, i.e.
$g_{w_2}^{-1} g_{w_1} \cdot v = v$, which implies $w_1 = w_2$. Moreover, $\varphi$ is surjective: for $g_w G_v \in G/G_v$, take $w := g_w \cdot v$, so $\varphi(w) = g_w G_v$. Thus, $\varphi$ is bijective. Since $G/G_v$ is countable, it follows that $V(X)$ is countable.
\end{proof}

Suppose $H<G$ are compactly generated TDLC groups such that $H$ is open. Let $K_H$ and $K_G$ be compact generating sets of $H$ and $G$, respectively. Let $U$ be a compact open subgroup of $G$. Then, $U_H=U\cap H$ is compact and open in $H$. Let $S_G$ and $S_H$ be finite symmetric subsets of $H$ and $G$ respectively such that $K_H\subset S_HU_H$ and $K_G\subset S_GU.$  Let $X_H=X_H(K_H,U_H,S_H)$ and $X_G=X_G(K_G,U,S_G)$ be Cayley-Abels graphs of $H$ and $G$, respectively. We denote by $d_{X_H}$ and $d_{X_G}$ the metric on $X_H$ and $X_G$, respectively. We now assume that $S_H\subset S_G$. Note that distinct cosets of $U_H$ in $H$ give distinct cosets of $U$ in $H$. Thus, we have an embedding $i:X_H\to X_G$ that sends the vertex $hU_H$ to the vertex $hU$. Thus, we can treat $X_H$ as a subgraph of $X_G$. We are now ready to prove the following:

\begin{lemma}\label{lemma-proper-embedding-subgraph}
    The map $i:X_H\to X_G$ is a proper embedding, i.e. there exists a function $\eta:[0,\infty)\to [0,\infty)$ with $\eta(n)\to \infty$ as $n\to \infty$ such that for all $h_1U_H,h_2U_H\in V(X_H)$, if $d_{X_G}(h_1U,h_2U)\leq n$ then $d_{X_H}(h_1U_H,h_2U_H)\leq \eta(n)$.
\end{lemma}

\begin{proof}
Define 
$$\eta(n) := \max \{ d_{X_H}(U_H, hU_H) : h \in H,\, d_{X_G}(U, hU) \leq n \}.$$
It is sufficient to prove that if $d_{X_G}(U, hU) \leq n$ then $d_{X_H}(U_H, hU_H) \leq \eta(n)$ which follows from the definition of $\eta.$
It remains to see that $\eta(n) \to \infty$ as $n\to\infty$. For some $M>0$, let if possible $\eta(n) \leq M$ for all $n$. Then, for every $N\in\mathbb N$ there exists $h_N \in H$ with $d_{X_H}(U_H, h_N U_H) \leq M$ but $d_{X_G}(U, h_N U) \leq N$. This yields infinitely many points in the ball of radius $M$ around $U_H$ in $X_H$. This gives a contradiction, as $X_H$ is locally finite, balls of finite radius contain finitely many vertices. 
\end{proof}

Now, we make a record of the following basic lemma.

\begin{lemma}\label{compact-open-infinite}
Open subgroups of a non-discrete locally compact group are infinite. In particular, the intersection of any two open subgroups of a locally compact group is infinite.
 \end{lemma}
 \begin{proof}
     Let $G$ be a locally compact group and $U$ be an open subgroup of $G$. Then, with the subspace topology, $U$ is also locally compact. Since singletons are closed, if $U$ is finite, then the topology on $U$ is discrete. This, in turn, implies that $G$ is discrete, and we arrive at a contradiction. 
 \end{proof}
For discrete groups, the theory of ends is well-known. In \cite{kron-moller}, the authors defined the space of ends for TDLC groups. We end this subsection with the following definition that turns out to be useful for us later.
 
 \begin{definition}\label{definition-rough-ends}
    Let $G$ be a compactly generated TDLC group. The space of {\em rough ends} of $G$ is defined as the end space of a Cayley-Abels graph of $G$. 
\end{definition}

Since Cayley-Abels graphs are unique up to quasiisometry and the locally finite quasiisometric graphs have homeomorphic end spaces, the above definition is well-defined.
 
\subsection{Hyperbolic TDLC groups and their boundaries}\label{subsection-gromov-hyp-tdld-and-boundaries} In \cite{gromov-hypgps}, Gromov introduced the notion of hyperbolic geodesic metric space and proved that hyperbolicity is an invariant of quasiisometry, see also \cite{bridson-haefliger}. It is well-known that any two Cayley-Abels graphs of a compactly generated TDLC group are quasiisometric \cite{kron-moller}. This motivates the following well-defined notion of a hyperbolic TDLC group. 
 
\begin{definition}\label{definition-hyperbolic-tdlc-group}
    A compactly generated TDLC group $G$ is said to be {\em hyperbolic} if some (hence any) Cayley-Abels graph of $G$ is a Gromov hyperbolic space. 
\end{definition}

In \cite{gromov-hypgps}, Gromov introduced the notion of a boundary of a hyperbolic geodesic metric space, which is now known as the Gromov boundary named after him. In the same essay, he proved that quasiisometric proper hyperbolic spaces have homeomorphic Gromov boundaries (see also \cite[III.H]{bridson-haefliger}). This motivates the notion of the Gromov boundary of a hyperbolic TDLC group. Let $G$ be a hyperbolic TDLC group. One then defines the \emph{Gromov boundary} $\partial G$ to be the Gromov boundary of some (hence any) Cayley-Abels graph $X$ of $G$: the set of equivalence classes of geodesic rays in $X$ starting from a base point $x_0$, i.e.
\[
\partial G := \partial X= \{\, [\gamma] \mid \gamma\text{ is a geodesic ray in } X \text{ such that } \gamma(0)=x_0 \},
\]
where two rays are equivalent if their Hausdorff distance is finite. Define $\overline{X}:=X\cup\partial X.$ Then, with a natural topology, $\overline{X}$ is a compact metrizable space, and $X$ is an open dense subset of $\overline{X}.$ For details, one is referred to \cite[III.H]{bridson-haefliger}.
\subsection{Relatively hyperbolic TDLC groups}\label{subsection-rel-hyp-tdrh-1}
The notion of relatively hyperbolic groups was introduced by Gromov for discrete groups \cite{gromov-hypgps}. Since then, it has been elaborated by several authors, for example, see \cite{farb-relhyp,bowditch-relhyp,drutu-sapir,osin-book,groves-manning}. To define relative hyperbolicity for TDLC groups, Arora-Pedroza \cite{arora-pedroja} took an approach due to Bowditch \cite{bowditch-relhyp} and Osin \cite{osin-book}. 
We first recall their definition of relative hyperbolic TDLC groups, some useful facts, which will be required later in the paper.

Let $G$ be a TDLC group, and $\mathcal H$ be a finite collection of open subgroups in $G$. A pair $(G,\mathcal H)$ is called a {\it proper pair} if no pair of distinct non-compact subgroups in $\mathcal H$ are conjugate in $G$ (page-832, \cite{arora-pedroja}).

\begin{definition}[Cayley-Abels graph of $G$ with respect to $\mathcal H$ \cite{arora-pedroja}]\label{definition-cayley-abels-graph-with-resepct-to}
Suppose $(G,\mathcal H)$ is a proper pair and  $G$ is acting cocompactly, discretely (pointwise stabilizers of cells are open subgroups) on a connected simplicial graph $X$. Then, $X$ is said to be {\em Cayley-Abels graph of $G$ with respect to $\mathcal H$} if
\begin{enumerate}
    \item[{$(1)$}] the edge $G$-stabilizers are compact,
    \item[{$(2)$}] the vertex $G$-stabilizers are either compact or conjugates of subgroups in $\mathcal H$,
    \item[{$(3)$}] every $H\in\mathcal H$ is the $G$-stabilizer of a vertex of $X$,
    \item[{$(4)$}] any pair of vertices of $X$ with the same $G$-stabilizer $H\in\mathcal H$ are in the same $G$-orbit if $H$ is non-compact.
    \end{enumerate}
	 \end{definition}
The group $G$ is said to be {\em compactly generated relative to $\mathcal H$} if there exists a compact set $K\subset G$ such that $K\cup (\bigcup\mathcal H)$ algebraically generates $G$.
\vspace{.2cm}

{\bf Existence of Cayley-Abels graph relative to $\mathcal H$:} Let $K$ be a compact subset of $G$ such that $G=\langle K\cup (\bigcup_{H\in\mathcal H}H)\rangle$. Since $K$ is compact there is a finite subset $S\subset G$ containing the identity of $G$ such that $K\subset SU.$
 Define a graph $X$ with $V(X):=G/U\cup(\bigcup_{H\in\mathcal H}G/H)$ and $E(X):=\{\{gU,gsU\}|g\in G, s\in S\}\cup\{\{gU,gH\}|g\in G, H\in\mathcal H\}$. Note that $G$ acts discretely, there are only finitely many orbits of vertices and edges on $X$, edge stabilizers are compact open, and vertex stabilizers are conjugates of $U$ or $H\in\mathcal H$. Other properties for being a Cayley-Abels graph also follow, see \cite[Proposition 4.6]{arora-pedroja}.
\vspace{.2cm}

 We collect a few properties of a Cayley-Abels graph of the proper pair $(\mathcal G,\mathcal H)$ in the form of the following proposition.

\begin{proposition}\label{proposition-properties-rel-cayley-graph}
    Let $(G,\mathcal H)$ be a proper pair. Then, we have the following:
    \begin{enumerate}
        \item \cite[Theorem F]{arora-pedroja} $G$ is compactly generated relative to $\mathcal H$ if and only if there exists a Cayley-Abels graph of $G$ relative to $\mathcal H$.
        \item \cite[Theorem H]{arora-pedroja} Suppose $X_1$ and $X_2$ are two Cayley-Abels graphs of $G$ with respect to $\mathcal H$. Then, $X_1$ and $X_2$ are quasiisometric, and $X_1$ is fine if and only if $X_2$ is fine.
    \end{enumerate}
\end{proposition}

The definition of a fine graph is also important to us.

\begin{definition}\label{definition-fine-graph}
    A simplicial graph $X$ is said to be {\em fine} if for any pair of vertices $u,v$ and $n\in\mathbb N$, there are finitely many embedded paths of length $n$ from $u$ to $v$.
\end{definition}
 
Now, we are ready to define the notion of relatively hyperbolic TDLC groups. 

\begin{definition}[TDRH-I]\label{definition-tdrh-1}
For a proper pair $(G,\mathcal H)$, the topological group $G$ is said to be hyperbolic relative to $\mathcal H$ if there exists a Cayley-Abels graph of $G$ with respect to $\mathcal H$ that is fine and hyperbolic. 
\end{definition}
From Proposition \ref{proposition-properties-rel-cayley-graph}, the previous definition does not depend on the choice of the Cayley-Abels graph of $G$ with respect to $\mathcal H$.

\subsection{Graphs of topological groups}\label{graphs-of-topological-groups}
In geometric group theory, the notion of graphs of groups is classical. One is referred to Serre's book on trees \cite{serre-trees} for a detailed account of definitions and results.
\begin{definition}\label{definition-graphs-of-groups}
Let $\mathcal Z$ be a graph. A {\em graph of topological groups} $(\mathcal G,\mathcal Z)$ over $\mathcal Z$ consists of the following data:
\begin{enumerate}
    \item[{$(1)$}] For each $v\in V(\mathcal Z)$, there is a topological group $G_v$ called the {\em vertex group}.
    \item[{$(2)$}] For each $e\in E(\mathcal Z)$, there is a topological group $G_e$ called the {\em edge group.}
    \item[{$(3)$}] For each edge $e\in E(\mathcal Z)$ with vertices $v$ and $w$, there are open topological embeddings $G_e\to G_v$ and $G_e\to G_w.$
\end{enumerate}
\end{definition}
The fundamental group of the graph of group $(\mathcal G,\mathcal Z)$ is denoted by $\pi_1(\mathcal G,\mathcal Z)$.
The next proposition shows that there is canonical topology on $\pi_1(\mathcal G,\mathcal Z).$
\begin{proposition}\textup{\cite[Proposition 8.B.9 and Proposition 8.B.10]{cornulier-harpe-book-metric-geometry}}\label{proposition-topology-graphs-of-groups}
There is a unique topology on $\pi_1(\mathcal G,\mathcal Z)$ such that the inclusion map $G_v\to \pi_1(\mathcal G,\mathcal Z)$ is an open topological embedding. Moreover, we have the following:
\begin{enumerate}
\item[{$(1)$}] If the edge groups of $(\mathcal G,\mathcal Z)$ are locally compact, then $\pi_1(\mathcal G,\mathcal Z)$ is locally compact.

\item[{$(2)$}] If the vertex groups of $(\mathcal G,\mathcal Z)$ are compactly generated then $\pi_1(\mathcal G,\mathcal Z)$ is compactly generated.
\end{enumerate}
\end{proposition}
In the forthcoming section, we prove combination theorems for topological groups with some extra properties. To prove such a theorem, it is sufficient to prove the theorem for an amalgamated free product and an HNN extension of topological groups. These are special cases of graphs of groups. In particular, if $\mathcal Z$ is an interval, $\pi_1(\mathcal G,\mathcal Z)$ is the amalgamated free product. If $\mathcal Z$ is a simple loop, then $\pi_1(\mathcal G,\mathcal Z)$ is an HNN extension. We now explain in detail these constructions.

Suppose $A,B$, and $C$ are topological groups such that $i_A:C\to A$ and $i_B:C\to B$ are topological isomorphisms onto open subgroups of $A$ and $B$, respectively. Then, we have an {\em amalgamated free product} $A\ast_C B$ given by the presentation $\langle S_A,S_B|R_A,R_B, i_A(c)=i_B(c) \text{ for all } c\in C \rangle$ if $\langle S_A|R_A\rangle$ and  $\langle S_B|R_B\rangle$ are presentations of $A$ and $B$, respectively. {\em Throughout the paper, we assume that the amalgamated free products are non-trivial, i.e. the edge group is not equal to both vertex groups}. If $C$ is a subgroup of $A$ and $\alpha$ is a topological isomorphism from $C$ to an open subgroup of $A$, then we have {\em HNN extension} $A\ast_C$ whose presentation is given by $\langle S_A,t|R_A, t^{-1}ct=\alpha(c) \text{ for all } c\in C \rangle$ if  $\langle S_A|R_A\rangle$ is a presentation of $A$. Let $G$ denote an amalgamated free product $A\ast_C B$ or an HNN extension $A\ast_C$. Then, by Proposition \ref{proposition-topology-graphs-of-groups}, there is a unique topology on $G$ such that the inclusions $A\to G, B\to G$, and $C\to G$ are topological isomorphisms onto open subgroups of $G$. {\em Throughout the paper, we use this topology on amalgamated free products and HNN extensions of topological groups.} We record the following lemma, which is relevant to us.

\begin{lemma}\label{lemma-tdlc-graph-of-groups-if-edge-is-tdlc}
    Suppose $G= A \ast_C B$ or $G = A \ast_C$. If $G$ is a TDLC group, then $A, B$, and $C$ are TDLC groups. Conversely, if $C$ is TDLC, then $G$ is TDLC.
\end{lemma}

\begin{proof}
Observe that a subgroup of a totally disconnected topological group is totally disconnected. Thus, if $G$ is TDLC, then $A,B$, and $C$ are TDLC. Conversely, suppose $C$ is TDLC. Then, by the definition of the topology on $G$, $C$ is embedded in $G$ as an open subgroup. By van Dantzig’s theorem, there is a neighborhood basis at the identity in $C$ consisting of compact open subgroups. Since $C$ is open in $G$, this neighborhood basis at the identity in $C$ gives a neighborhood basis of compact open subgroups at the identity in $G$.  Hence, $G$ is a TDLC group.
\end{proof}

From Lemma \ref{lemma-tdlc-graph-of-groups-if-edge-is-tdlc}, it follows that if the edge groups of $(\mathcal G,\mathcal Z)$ are TDLC then $\pi_1(\mathcal G,\mathcal Z)$ is TDLC.
\vspace{.2cm}

{\bf Bass-Serre trees of amalgamated free product and HNN extension:} Let $G = A \ast_{C} B$. The construction of the Bass--Serre tree of $G$ is classical \cite{serre-trees}. For the sake of completeness, we explain the construction as follows.

Let $\tau$ be a unit interval with vertices $v_{A}$ and $v_{B}$. Define an equivalence relation $\sim$ on $G \times \tau$ induced by the equivalences.
 $$(g_{1},v_{A}) \sim (g_{2},v_{A}) \text{ if } g_{1}^{-1} g_{2} \in A,$$
 $$(g_{1},v_{B}) \sim (g_{2},v_{B}) \text{ if } g_{1}^{-1} g_{2} \in B,$$
 $$(g_{1},t) \sim (g_{2},t) \text{ if } g_{1}^{-1} g_{2} \in C, t \in \tau.$$
 Then $G \times \tau /_\sim$ is called the {\em Bass-Serre tree} $T$ of $G$. 

 If $G=A\ast_C$ is an HNN extension. Then, the Bass-Serre tree $T$ of $G$ is defined as follows:
 
 The vertex set of $T$ is the set of left cosets of $A$ in $G$. The set of edges of $T$ is the set of left cosets of $C$ in $G$. For all $g \in G$, vertices $gA$ and $gtA$ are connected by the edge $gC$.
 
\vspace{.2cm} 
{\bf Notation:} Let $T$ be the Bass-Serre tree of amalgamated free product or HNN extension. Let $\pi$ denote the natural projection from $T$ to $\tau$ or a simple loop. For each vertex $v$ of $T$, we denote by $G_v$ the $G$-stabilizer of $v$. 
\section{A combination theorem for relatively hyperbolic TDLC groups}\label{section-rel-hyp-combination-theorem}

The main goal of this section is to prove Theorem \ref{rel-combination-theorem-amalgam}. To prove that, it is sufficient to consider the amalgamated free product and HNN extension case. Throughout this section, all relatively hyperbolic groups are TDRH-I (Definition \ref{definition-tdrh-1}). 

\subsection{Amalgamated free product case.}\label{subsection-rel-hyp-amalgam-combination-construction}  Let $G$ be as in Theorem \ref{rel-combination-theorem-amalgam}. Let $i_A:C\to A$ and $i_B:C\to B$ be topological isomorphisms onto open subgroups of $A$ and $B$, respectively. As discussed in Subsection \ref{subsection-rel-hyp-tdrh-1}, let $X_A$ be the Cayley-Abels graph of $A$ with respect to $\mathcal H_A$ such that the set of cosets of $i_A(C)$ in $A$ is a subset of the vertices of $X_A$. Similarly, let $X_B$ be the Cayley-Abels graph of $B$ with respect to $\mathcal H_B$ such that the set of cosets of $i_B(C)$ in $B$ is a subset of the vertices of $X_B$. Define a graph $Y$ as the union of $X_A$, $X_B$, and $\tau$, where the coset $i_A(C)$ is identified with $v_A$ and the coset $i_B(C)$ is identified with $v_B$. Then, we define the graph $X$ as an equivalence relation on $G\times Y$ induced by 
$$(g_1,y_1)\sim(g_2,y_2) \text{ if }y_1,y_2\in X_A,\; g_2^{-1}g_1\in A \text{ and } g_2^{-1}g_1y_1=y_2,$$
$$(g_1,y_1)\sim(g_2,y_2) \text{ if }y_1,y_2\in X_B,\; g_2^{-1}g_1\in B \;\text{and}\; g_2^{-1}g_1y_1=y_2.$$
We have a natural $1$-Lipschitz subjective map from $X$ to $T$ by sending
$(g,X_A)$ to $(g,v_A)$ and $(g,X_B)$ to $(g,v_B)$ for $g\in G$. Also, there is a natural action of $G$ by left on $X.$

\begin{lemma}\label{lemma-rel-cayley-abels-graph-amalgam}
    The graph $X$ is a Cayley-Abels graph of $G$ with respect to $\mathcal H_A\cup\mathcal H_B$.
\end{lemma}
\begin{proof}
From the construction of $X$, we see that $X$ is a connected graph. Note that every vertex of $X$ belongs to $(g,X_A)$ or $(g,X_B)$ for some $g\in G$. Since the action of $A$ on $X_A$ and $B$ on $X_B$ are cocompact, the action of $G$ on $X$ is cocompact. Clearly, the vertex stabilizers are either compact or conjugates of subgroups in $\mathcal H$. The edge stabilizers are also compact (note that the $G$-stabilizer of $\tau$ is $C$). Every $H\in\mathcal H$ is either the $G$-stabilizer of a vertex in $X_A$ or a $G$-stabilizer of a vertex in $X_B$. The last property for being a Cayley-Abels graph with respect to $\mathcal H_A\cup \mathcal H_B$ is also satisfied. This completes the proof of the lemma.
\end{proof}

Proof of Theorem \ref{rel-combination-theorem-amalgam}:
    Let $A$ and $B$ be two hyperbolic groups relative to $\mathcal H_A$ and $\mathcal H_B$, respectively. Then, $X_A$ and $X_B$ are hyperbolic and fine. Thus, $X$ is a tree of hyperbolic spaces such that the edge spaces are the singletons. By \cite{BF}, $X$ is a hyperbolic graph. Since each vertex space is a fine graph, $X$ is also a fine graph. Hence, $G$ is hyperbolic relative to $\mathcal H_A\cup\mathcal H_B$. 
\qed

Note that, if $G$ is hyperbolic relative to $\mathcal H$, and $\mathcal H$ is empty, then $G$ is compactly generated and the Cayley-Abels graph of $G$ with respect to $\mathcal H$ quasiisometric to a Cayley-Abels graph of $G$ (\cite[Theorem 2.9]{kron-moller}), and hence, $G$ becomes a hyperbolic TDLC group.

\begin{corollary}\label{corollary-combination-amalgam-hyp-tdlc}
    Suppose $G=A\ast_C B$ is an amalgamated free product such that $A,B$ are compactly generated TDLC groups and $C$ is compact. Then, $G$ is hyperbolic if and only if $A$ and $B$ are hyperbolic. 
\end{corollary}
\begin{proof}
    Suppose $A$ and $B$ are hyperbolic groups. Then, $A$ and $B$ are both hyperbolic relative to an empty collection of subgroups, respectively. By Theorem \ref{rel-combination-theorem-amalgam}, $G$ is hyperbolic relative to an empty collection of subgroups. Thus, $G$ is a hyperbolic group. Conversely, suppose $G$ is hyperbolic. Let $X$ be the graph as constructed above by taking $\mathcal H_A=\phi$ and $\mathcal H_B=\phi.$ Note that, by \cite[Theorem 2.9]{kron-moller}, $X$ is quasiisometric to a Cayley-Abels graph of $G$. Since $G$ is hyperbolic, $X$ is a hyperbolic space. By construction of $X$, the inclusion maps $X_A\to X$ and $X_B\to X$ are isometric embeddings. Hence, $A$ and $B$ are hyperbolic.
\end{proof}

 \begin{example}\label{example-amalgam-hyperbolic} Let $\mathbb Q_p$ denotes the field of $p$-adic numbers, and $\mathbb Z_p$ denote the group of $p$-adic integers.
 It is well-known that ${\rm SL}(2,\mathbb Q_p)={\rm SL}(2,\mathbb Z_p)\ast_C {\rm SL}(2,\mathbb Z_p)$, where $C$ is a common compact open subgroup of ${\rm SL}(2,\mathbb Z_p)$ \cite{serre-trees}. Since ${\rm SL}(2,\mathbb Z_p)$ is compact, it is hyperbolic. Although it is known that ${\rm SL}(2,\mathbb Q_p)$ is hyperbolic, however, from Theorem \ref{rel-combination-theorem-amalgam}, we can also see that ${\rm SL}(2,\mathbb Q_p)$ is hyperbolic. Moreover, Cayley-Abels graphs of ${\rm SL}(2,\mathbb Q_p)$ are quasiisometric to the Bass-Serre tree of the splitting.
 \end{example}


\subsection{HNN extension case.}\label{subsection-rel-hyp-hnn-combination-construction} Let $G$ be as in Theorem \ref{theorem-rel-combination-theorem-hnn}.
Let $X_A$ be the Cayley-Abels graph of $A$ with respect to $\mathcal H$ such that the set of the left cosets of $C$ in $A$ is a subset of the set of vertices of $X_A$. Consider the graph $Y$ induced by the following equivalence relation on $G\times X_A$ :
$$(g_1,x_1)\sim(g_2,x_2) \text{ if }x_1,x_2\in X_A,\;g_2^{-1}g_1\in A \text{\;and } g_2^{-1}g_1x_1=x_2.$$
Denote the equivalence class $[g,x]$ by $gx$ and $X_{gA}$ be the subgraph spanned by the vertices in $\{gx \in Y: x\in X_A \}$. These are the connected components of $Y$, which are isometric to $X_A$. Now, we form a new graph $X$ using the graph $Y$ as follows:

For all $g \in G$, the vertices $gA$ and $gtA$ are connected in the Bass-Serre tree $T$. Accordingly, we join the vertices $gC \in X_{gA}$ and $gtC \in X_{gtA}$. We have a natural $1$-Lipschitz projection map from $X$ to $T$ that sends $X_{gA}$ to $gA$. Also, $X$ is equipped with the natural action of $G$ on the left. From the construction of $X$, we have the following, whose proof is similar to the proof of Lemma \ref{lemma-rel-cayley-abels-graph-amalgam}.

\begin{lemma}\label{rel-hyp-cayley-ables-graph-hnn}
    The graph $X$ is a Cayley-Abels graph of $G$ with respect to $\mathcal H$.          \qed
\end{lemma}

Now, by following the proof of Theorem \ref{rel-combination-theorem-amalgam}, the proof of the following theorem follows. This completes the proof of Theorem \ref{theorem-rel-combination-theorem-hnn}. For the convenience of the reader, we restate Theorem \ref{theorem-rel-combination-theorem-hnn}.

\begin{theorem}
Suppose $A$ is a compactly generated TDLC group relative to collections of open subgroups $\mathcal H$.
Suppose $G=A\ast_C$ is an HNN extension such that $A$ is hyperbolic relative to $\mathcal H$ and $C$ is compact. Then, $G$ is hyperbolic relative to $\mathcal H.$            \qed
\end{theorem}   

The idea of the proof of the following corollary is the same as the proof of Corollary \ref{corollary-combination-amalgam-hyp-tdlc}. Hence, we omit the details.

\begin{corollary}\label{corollary-combination-hnn-hyperbolic-tdlc}
Suppose $G=A_{\ast_C}$ where $A$ is a compactly generated TDLC group and $C$ is a compact open subgroup of $G$. Then, $G$ is hyperbolic if and only if $A$ is hyperbolic.    \qed
\end{corollary}

\begin{example}\label{example-hnn-hyperbolic}
   \begin{enumerate}
    \item[{$(1)$}] Let $\alpha$ be an automorphism of ${\rm SL}(2,\mathbb Z_p)$, and $G=\mathbb Z\ltimes_{\alpha} {\rm SL}(2,\mathbb Z_p)$ be the semidirect product of ${\rm SL}(2,\mathbb Z_p)$. Since $G$ is also an HNN extension of ${\rm SL}(2,\mathbb Z_p)$ over itself, $G$ is hyperbolic by Theorem \ref{theorem-rel-combination-theorem-hnn}. Moreover, the Cayley-Abels graph of $G$ is quasiisometric to the real line.
    
    \item[{$(2)$}] Let $\alpha\in \mathrm{Aut}(\mathbb Q_p)$ be an automorphism of multiplication by $p$. Let $G=\mathbb Z\ltimes_{\alpha} \mathbb Q_p$ be the semidirect product. Also, $G=\mathbb Z_p\ast_{\mathbb Z_p\simeq p\mathbb Z_p}$ \cite[Example 8.C.17(1)]{cornulier-harpe-book-metric-geometry}. Therefore, $\mathbb Z\ltimes_{\alpha}\mathbb Q_p$ is compactly generated and is hyperbolic by Theorem \ref{theorem-rel-combination-theorem-hnn}. 
    
\item[{$(3)$}] For $d\geq 2$, let $T_d$ denotes the $d$-regular tree. Let $G =\mathrm{Aut}(T_d)$ be the automorphism group of $T_d$. Since $T_d$ is a Cayley-Abels graph for $G$, it follows that $G$ is hyperbolic. However, one can also view $G$ as a multiple HNN extension of the stabilizer of a vertex in $T_d$, see Example \ref{example-aut(t_d)} for details. Therefore, $G$ is hyperbolic by Theorem
    \ref{theorem-rel-combination-theorem-hnn}.
\end{enumerate}
\end{example}

\begin{remark}
    A subgroup of a TDLC hyperbolic group need not be hyperbolic. For example, the semidirect product 
$\mathbb Z \ltimes \mathbb Q_p$
is hyperbolic (see Example~\ref{example-hnn-hyperbolic}), yet its normal subgroup \(\mathbb Q_p\) is not hyperbolic, because it is not compactly generated. In the discrete setting, there are well-known examples. This gives an example in a non-discrete setting.
 \end{remark}

  \begin{remark}\label{remark-Z}
     $(1)$ Suppose $G$ is a discrete hyperbolic group. Then, it cannot contain a subgroup isomorphic to $\mathbb Z\times \mathbb Z$. In contrast, within the TDLC setting, ${\rm SL}(2,\mathbb Z_p)\times {\rm SL}(2,\mathbb Z_p)$ is compact, and hence hyperbolic, yet it contains a copy of $\mathbb Z\times \mathbb Z$. A similar phenomenon occurs with ${\rm SL}(2,\mathbb Z_p)\times \mathbb Z$ which, as noted in 
     Example \ref{example-hnn-hyperbolic}(1) is also hyperbolic, and it contains a copy of $\mathbb Z\times \mathbb Z.$ These examples show that, in the TDLC setting, having a subgroup isomorphic to $\mathbb{Z}\times \mathbb{Z}$ is not an obstruction to hyperbolicity.

     $(2)$ It is known that ${\rm SL}(2,\mathbb Q_p)$ acts on $(p+1)$-regular tree $T_{p+1}$ which nothing but the Bruhat-Tits building for ${\rm SL}(2,\mathbb Q_p)$ \cite{serre-trees}. Let $T_1$ and $T_2$ be two copies of $T_{p+1}$. Then, ${\rm SL}(2,\mathbb Q_p)\times {\rm SL}(2,\mathbb Q_p)$ acts properly cocompactly on $T_1\times T_2$. Since $T_1\times T_2$ is not hyperbolic, ${\rm SL}(2,\mathbb Q_p)\times{\rm SL}(2,\mathbb Q_p)$ is also not hyperbolic. However, if $A={\rm SL}(2,\mathbb Q_p)\times{\rm SL}(2,\mathbb Q_p)$ and $C$ is any compact open subgroup of ${\rm SL}(2,\mathbb Q_p)\times{\rm SL}(2,\mathbb Q_p)$, then the amalgamated free product $A\ast_C \overline{A}$ is hyperbolic relative to $\{A,\overline{A}\}$ where $\overline{A}$ denotes a copy of ${\rm SL}(2,\mathbb Q_p)\times{\rm SL}(2,\mathbb Q_p)$.
\end{remark} 
 
\subsection{The general case.}
By induction and using Theorem \ref{rel-combination-theorem-amalgam} and Theorem \ref{theorem-rel-combination-theorem-hnn}, one can prove Theorem \ref{theorem-rel-hyp-combination-general-case}. Hence, we skip the details.

\begin{theorem}\label{theorem-rel-hyp-combination-general-case}
Let $\mathcal Z$ be a finite graph, and let $(\mathcal G,\mathcal Z)$ be a graph of groups such that the following hold:
\begin{enumerate}
    \item[{$(i)$}] For each $v\in V(\mathcal Z)$, let $\mathcal H_v$ be a finite collection of subgroups of the vertex group $G_v$, and $(G_v,\mathcal H_v)$ be a proper pair. Suppose $G_v$ is hyperbolic relative to $\mathcal H_v$.
    \item[{$(ii)$}] For each $e\in E(\mathcal Z)$, suppose the edge group $G_e$ is compact subgroup of the adjacent vertex groups.
\end{enumerate}
    Then, the fundamental group of $(\mathcal G,\mathcal Z)$ is hyperbolic relative to $\bigcup_{v\in V(\mathcal Z)}\mathcal H_v$.
    \qed
\end{theorem}

In particular, we deduce if $(G,\mathcal Z)$ is a graph of groups over $\mathcal Z$ such that each vertex group is hyperbolic, and each edge group is a compact subgroup of the adjacent vertex groups, then the fundamental group of $(G,\mathcal Z)$ is a hyperbolic TDLC group.


\section{Compactly generated relatively hyperbolic TDLC groups}

Motivated by the work of Farb \cite{farb-relhyp}, Bowditch \cite{bowditch-relhyp}, and \cite{groves-manning} on discrete relatively hyperbolic groups, in this section, we introduce two notions of relative hyperbolicity for a compactly generated TDLC group. Then, we show that all three notions of relative hyperbolicity discussed in this paper are equivalent (Theorem \ref{theorem-rh-equivalence}). We conclude the section by giving a proof of Theorem \ref{Theorem-rel-hyp-intersection-is-compact}.

\subsection{Relative hyperbolicity in the sense of Bowditch and Groves-Manning.}  \label{subsection-tdrh-II} 
We start here by recalling the definition of a combinatorial horoball.

\begin{definition}\label{augmented-Cayley-Abels-graph}
		 Let $\Gamma$ be a graph. The {\em combinatorial horoball based at $\Gamma$}, denoted by $\mathcal B(\Gamma)$, is a graph defined as follows:
		\begin{enumerate}
			\item[{$(1)$}] $\mathcal B^{(0)}=\Gamma^{(0)}\times (\{0\}\cup \mathbb{N})$.
			\item[{$(2)$}] $\mathcal{B}^{(1)}$ contains the following three types of edges:
			
			$(a)$ If $v$ and $w$ are joined by an edge in $\Gamma$, then there is an edge connecting $(v,0)$ and $(w,0)$.
			
			$(b)$ For $k>0$ and $v,w\in \Gamma^{(0)}$, if $0<d_{\Gamma}(v,w)\leq 2^k$ then there is a single edge connecting $(v,k)$ and $(w,k)$.
			
			$(c)$ For $k\geq 0$ and $v\in \Gamma^{(0)}$, there is an edge joining $(v,k)$ and $(v,k+1)$.
		\end{enumerate}
	\end{definition}
    
\begin{remark}\label{remark-2}
	\begin{enumerate}
			\item[{$(1)$}] As the full subgraph of $\mathcal{B}(\Gamma)$ containing the vertices $\Gamma^{(0)}\times\{0\}$ is isomorphic to $\Gamma$, we may think of $\Gamma$ as a subset of $\mathcal B(\Gamma)$.
	
	\item[{$(2)$}] If $\Gamma$ is locally finite, then it follows that $\mathcal B(\Gamma)$ is also locally finite.
    \end{enumerate}
\end{remark}

  Let G be a compactly generated TDLC group, and $\mathcal H=\{H_1,\dots, H_n\}$ be a finite collection of compactly generated open subgroups of $G$. Let $U$ be a compact open subgroup of $G$. Let $K$ be a compact generating set of $G$, and let $K_i$ be a compact generating set of $H_i$ for all $i$. Note that $U\cap H_i$ is an infinite compact open subgroup of $H_i$ for each $i$ by Lemma \ref{compact-open-infinite}. Then, there exists a finite symmetric set $S\subset G$ containing the identity element such that $K\subset SU$. Also, for each $i$, there is a finite symmetric set $S_i\subset H_i$ containing identity such that $K_i\subset S_i(U\cap H_i)$. We assume that $S_i\subset S$ for all $i$. Thus, the Cayley-Abels graph $Y_i$ of $H_i$ is embedded in the Cayley-Abels graph $X$ of $G$. Moreover, this is a proper embedding by Lemma \ref{lemma-proper-embedding-subgraph}. Therefore, we can treat $Y_i$ as a subgraph of $X$.

\begin{definition}\label{definition-augmented-space}
	  For each $i\in\{1,2,\dots,n\}$, let $T_i$ be a left transversal for $H_i$ in $G$. For each $i$ and each $t\in T_i$, let $Y_{i,t}$ be the full subgraph of the Cayley-Abels graph $X$ containing the vertices $tY_i$. Each $Y_{i,t}$ is isomorphic to the Cayley-Abels graph of $H_i$. Then the {\em augmented Cayley-Abels graph} of $G$ is the graph
	$$X^h(K,S,U)=X \cup (\{\mathcal{B}(Y_{i,t}):t\in T_i, 1\leq i\leq n \})$$
where the graphs $Y_{i,t}\subset X$ and $Y_{i,t}\subset \mathcal{B}(Y_{i,t})$ are identified as suggested in Remark \ref{remark-2}.
\end{definition}
Whenever the choices of $K, S$, and $U$ are clear from the context, we will write $X^h$ instead of $X^h(K,S,U)$. Note that the graph $X^h$ is locally finite. The following proposition shows, up to quasiisometry, that Definition \ref{augmented-Cayley-Abels-graph} does not depend on the choice of compact generating sets and compact open subgroups.

\begin{proposition}\label{well-definedness-augmented-cayley-graph}
    Suppose $K$ and $K'$ are compact generating sets of $G$, and $K_i, K_i'$ are compact generating sets of $H_i$ for each $i$. Let $U$ and $U'$ be compact open subgroups of $G$, and $ S$ and $ S'$ be finite symmetric generating sets containing the identity such that $K\subset SU$ and $K'\subset S'U'$. Suppose $S$ and $S'$ contain finite symmetric sets $S_i$ and $S_i'$ for $K_i$ and $K_i'$, respectively. Then, the augmented Cayley-Abels graphs $X^h(K,S,U)$ and $X'^h(K',S',U')$ are quasiisometric.
\end{proposition}

\begin{proof}
    Let $X(K,S,U)$ and $X'(K',S',U')$ be the Cayley-Abels graphs of $G$. Then, by \cite[Theorem 2.7$^{+}$]{kron-moller}, there is a quasiisometry $\psi:X(K,S,U)\to X'(K',S',U')$. Let $Y_{i,t}$ and $Y_{i,t}'$ denote the various copies of the Cayley-Abels graph of $H_i$ in $X(K,S,U$ and in $X(K',S',U')$, respectively. From the definition of $\psi$, it follows that $\psi(Y_{i,t})\subset N_D(Y_{i,t}')$ and $\psi^{-1}(Y_{i,t}')\subset N_D(Y_{i,t})$ for a uniform constant $D\geq 0.$ By \cite[Theorem 1.2]{hruska-healy}, it follows that $\psi$ induces a quasiisometry from $X^h$ to $X'^h$ (see also \cite[Theorem 2.8]{mackay-sisto-maps-boundaries}). This completes the proof of the proposition.
\end{proof}

\begin{definition}[TDRH-II]\label{definition-tdrh-II}
	 Let $(G,\mathcal H)$ be a proper pair. The group $G$ is said to be {\em hyperbolic relative to $\mathcal{H}$} if the augmented Cayley-Abels graph $X^h$ of $G$ is a Gromov hyperbolic space.
	In this situation, we also say that the proper pair $(G,\mathcal{H})$ is a relatively hyperbolic group.
\end{definition}

Given a proper hyperbolic geodesic metric space, one can define the Gromov boundary associated to it \cite{gromov-hypgps,bridson-haefliger}. Bowditch \cite{bowditch-relhyp} generalizes this notion in the context of discrete relatively hyperbolic groups. In a similar manner, we define the Bowditch boundary of a relatively hyperbolic TDLC group.

\begin{definition}\label{definition-bowditch-boundary}
	 Suppose $G$ is a compactly generated TDLC group and $G$ is hyperbolic relative to $\mathcal{H}$. The {\em Bowditch boundary} of $G$ with respect to $\mathcal H$ is defined to be the Gromov boundary of the augmented Cayley-Abels graph $X^h$ of $G$. It is denoted by $\partial_{rel}(G)$.
\end{definition}

Since $X^h$ is locally finite, $\overline{X^h}:= X^h\cup \partial_{rel}(G)$ is a compact metrizable space (see \cite[III.H, Exercise 3.18(4)]{bridson-haefliger}, \cite[Chapter 7]{GhH}).

\subsection{Relative hyperbolicity in the sense of Farb}\label{tdrh-2} In this subsection, we adapt the definition of relative hyperbolicity due to Farb \cite{farb-relhyp} in our setup. We retain the notation from Subsection \ref{subsection-tdrh-II}.

\begin{definition}[Coned-off Cayley-Abels graph] \label{definition-coned-off-graph}
     Let $X$ be the Cayley-Abels graph of $G$, and $Y_i$ be the Cayley-Abels graph of $H_i$. For each $i$ and each $t\in T_i$, let $Y_{i,t}$ be the full subgraph of the Cayley-Abels graph $X$ containing the vertices $tY_i$. Form a new graph $\hat{X}(K,S,U)$, called the {\em coned-off Cayley-Abels graph with respect to $\mathcal H$}, as follows. For each $Y_{i,t}$, add a new vertex, called the {\em cone point}, $v(Y_{i,t})$ to $X$ and add an edge of length $1$ from this new vertex to each element of $Y_{i,t}$.
\end{definition}

When it is clear from the context, we shall write $\hat{X}$ in place of $\hat{X}(K, S, U).$ In the following proposition, we show that the coned-off Cayley-Abels graphs for $G$ are unique up to quasiisometry.

\begin{proposition}\label{propsition-welldefined-ferb}
    Suppose $K$ and $K'$ are compact generating sets of $G$, and $K_i, K_i'$ are compact generating sets of $H_i$ for each $i$. Let $U$ and $U'$ be compact open subgroups of $G$, and $ S$ and $ S'$ be two finite symmetric generating sets containing the identity such that $K\subset SU$ and $K'\subset S'U'$. Suppose $S$ and $S'$ contain finite symmetric sets $S_i$ and $S_i'$ for $K_i$ and $K_i'$, respectively. Then, the coned-off Cayley-Abels graphs $\hat{X}(K,S,U)$ and $\hat{X'}(K',S',U')$ are quasiisometric.
\end{proposition}

\begin{proof}
    By \cite[Theorem 2.7$^{+}$]{kron-moller}, there exists a quasiisometry $\psi:X(K,S,U)\to X'(K,S',U')$ such that $\psi(Y_{i,t})\subset N_D(Y_{i,t}')$ and $\psi^{-1}(Y_{i,t}')\subset N_D(Y_{i,t})$ for a uniform constant $D\geq 0.$ Also, note that the maps $Y_{i,t}\to X(K,S,U)$ and $Y'_{i,t}\to X'(K',S',U')$ are proper embeddings by Lemma \ref{lemma-proper-embedding-subgraph}. Let $\psi^h:X^h(K,S,U)\to X'^h(K',S',U')$ be the induced quasiisometry from Proposition \ref{well-definedness-augmented-cayley-graph}. Hence, by \cite[Lemma 9.34(3)]{kapovich-sardar-book} (see also \cite[Lemma 1.2.31]{pal-thesis}), $\psi$ induces a quasiisometry from $\hat{X}(K,S,U)$ to $\hat{X'}(K',S',U').$ This completes the proof of the proposition.
\end{proof}

\begin{remark}\label{remark-another-proof-welldefine-coned-off} By Lemma \ref{lemma-example-of-coned-off-cayley-graph}, we see that $\hat{X}(K,S,U)$ and $\hat{X'}(K',S',U')$ are Cayley-Abels graphs with respect to $\mathcal H$. Hence, by \cite[Theorem H]{arora-pedroja}, $\hat{X}(K,S,U)$ is quasiisometric to $\hat{X'}(K',S',U').$
\end{remark}

We are ready to define the following.

\begin{definition}[TDRH-III]\label{definition-tdrh-III}
Let $(G,\mathcal H)$ be a proper pair. The group $G$ is said to be {\em hyperbolic relative to $\mathcal H$} if $\hat{X}$ is a hyperbolic graph and it satisfies the bounded penetration property (BPP).
\end{definition}

For the definition of BPP, one is referred to \cite{farb-relhyp}. 
\subsection{Equivalence of the definitions}
We start here by noting the following:
\begin{lemma}\label{lemma-example-of-coned-off-cayley-graph}
    Let $(G,\mathcal H)$ be a proper pair as in Subsection \ref{subsection-tdrh-II}. Then, the coned-off Cayley-Abels graph of $G$ with respect to $\mathcal H$ is a Cayley-Abels graph of $G$ with respect to $\mathcal H$.
\end{lemma}
\begin{proof}
    Let $X$ be the Cayley-Abels graph of $G$. The natural action of $G$ on $X$ induces a natural action of $G$ on $\hat{X}$. Since the action of $G$ on $X$ is cocompact and the set of cone points has finitely many orbits, the action of $G$ on $\hat{X}$ is also cocompact. Clearly, the edge $G$-stabilizers are compact. Note that the $G$-stabilizers of the cone points are conjugate to elements in $\mathcal H$. Thus, the conditions (2) and (3) of Definition \ref{definition-cayley-abels-graph-with-resepct-to} hold. Since there are no two vertices of $\hat{X}$ with the same $G$-stabilizer, the last condition also holds. Hence, we are done.
\end{proof}

We now show the equivalence of all three definitions of relative hyperbolicity for compactly generated groups.

\begin{theorem}\label{theorem-rh-equivalence}
Let G be a compactly generated TDLC group, and $\mathcal H=\{H_1,\dots, H_n\}$ be a finite collection of compactly generated open subgroups of $G$. Then, the following are equivalent:
\begin{enumerate}
    \item[{$(1)$}] TDRH-I.
    \item[{$(2)$}] TDRH-II.
    \item[{$(3)$}] TDRH-III.
\end{enumerate}
\end{theorem}
\begin{proof}

    [TDRH-II$\iff$TDRH-III] If the pair $(G,\mathcal H)$ is TDRH-II relatively hyperbolic, then $X$ is hyperbolic relative to $\mathcal Y=\{Y_{i,t}:t\in T_i \text{ for all i }\}$ in the sense of Definition 3.4 of \cite{sisto-metric-relative-hyperbolicity}. Then, by \cite[Theorem 1.1]{sisto-metric-relative-hyperbolicity}, $X$ is hyperbolic relative to $\mathcal Y=\{Y_{i,t}:t\in T_i \text{ for all $i$ and $1\leq i\leq n$}\}$ in the sense of Definition 3.8 of \cite{sisto-metric-relative-hyperbolicity}. Since these two definitions are equivalent, we are done.

    [TDRH-III$\implies$ TDRH-I] Suppose $(G,\mathcal H)$ is TDRH-III relatively hyperbolic. Then, $\hat{X}$ is hyperbolic and satisfies the BPP property. Since by Lemma \ref{lemma-example-of-coned-off-cayley-graph}, $\hat{X}$ is a Cayley-Abels graph with respect to $\mathcal H$. Thus, to check TDRH-III, it remains to check that $\hat{X}$ is a fine graph. Since $X$ is locally finite and $\hat{X}$ satisfies the BPP condition, it follows from the proof of \cite[Proposition 1, p.81]{dahmani-thesis} that $\hat{X}$ is a fine graph.

    [TDRH-I$\implies$ TDRH-III] Suppose $(G,\mathcal H)$ is TDRH-I relatively hyperbolic. Then, there exists a Cayley-Abels graph with respect to $\mathcal H$ such that it is hyperbolic and fine. Since $\hat{X}$ is a Cayley-Abels graph with respect to $\mathcal H$, by Proposition \ref{proposition-properties-rel-cayley-graph}, $\hat{X}$ is hyperbolic and fine. To prove TDRH-III, it remains to check that $\hat{X}$ satisfies the BPP property. This is the content of \cite[Lemma 5, p.82]{dahmani-thesis}. Hence, we are done.
\end{proof}
\vfill
We now provide a proof of Theorem \ref{Theorem-rel-hyp-intersection-is-compact}.
\vspace{.2cm}

Proof of Theorem \ref{Theorem-rel-hyp-intersection-is-compact}:
  Since $G$ is relatively hyperbolic with respect to $\mathcal H $, by TDRH-I, there exists a Cayley-Abels graph of $G$ with respect to $\mathcal H $ which is hyperbolic and fine. Since the coned-off Cayley-Abels graph $\hat{X}$ of $G$ is quasiisometric to the Cayley-Abels graph of $G$ with respect to $\mathcal H$, $\hat{X}$ is a fine graph Proposition \ref{proposition-properties-rel-cayley-graph}. Without loss of generality, we can choose $U$ a compact open subgroup such that $U\subset\bigcap_{i=1}^nH_i$.

  For $(1)$, we show that the index of $U$ in $H_i\cap H_j$ is finite. Suppose not. Then, the Cayley-Abels graphs $Y_i$ and $Y_j$ have infinitely many common vertices. We enumerate these vertices as $\{h_k\}$. Let $v_i$ and $v_j$ denote the cone points corresponding to $Y_i$ and $Y_j$ in $\hat{X}$. By definition of $\hat{X}$, each $h_k$ is joined to $v_i$ and $v_j$ by an edge of length $1$. Now, it is clear that there are infinitely many simple loops of length $4$ in $\hat{X}$ containing the edge from $v_i$ to $h_1$. This gives a contradiction to the fineness of $\hat{X}$. Hence $[H_i\cap H_j:U]<\infty.$ Since $U$ is compact, $H_i\cap H_j$ is compact.

  For $(2)$ we show that $gUg^{-1}\cap U$ has finite index in $gH_ig^{-1}\cap H_j$. Suppose not. Since $gUg^{-1}\cap U$ has finite index in $gUg^{-1}$ as well as in $U$, so $gUg^{-1}\cap U$ has infinitely many distinct cosets representatives in $gH_ig^{-1}\cap H_j$ lying outside $gUg^{-1}\cap U$. These representatives also give coset representatives of $U$ in $gH_ig^{-1}\cap H_j$. Thus, $gY_i$ and $Y_j$ have infinitely many common vertices. By the same logic as in (1), we get a contradiction. This proves $(2)$.

  Let $X$ be a Cayley-Abels graph of $G$ and let $Y_i$ be a Cayley-Abels graph of $H_i$ such that $Y_i$ is embedded in $X_i$. Since $(G,\mathcal H)$ is relatively hyperbolic, $X$ is hyperbolic relative to $\{Y_i\}$. As each $H_i$ is hyperbolic, $Y_i$ is a hyperbolic graph. Thus, by \cite[Theorem 1]{hamnestadtrelative}, $X$ is hyperbolic. Thus, $G$ is a hyperbolic TDLC group. This proves $(3).$
\qed
\section{Bowditch boundaries of relatively hyperbolic amalgamated free products and HNN extensions of TDLC groups}\label{section-bowditch-boundaries-rel-hyp-amalgam-hnn}

Throughout this section, for relatively hyperbolic groups, we use Definition \ref{definition-tdrh-II}. The main goal of this section is to prove that the topology of the Bowditch boundary of a relatively hyperbolic amalgamated free product or HNN extension is uniquely determined by the topology of the Bowditch boundary of the vertex groups. For that, first of all, we give a construction of the Bowditch boundary for these groups, and then, we use the dense amalgam technique from \cite{swiatkowski-dense-amalgam} to conclude the result.

\subsection{Amalgamated free product case.}\label{bowditch-boudary-construction-amalgam} We start here by describing our setup.
Let $A$ and $B$ be two compactly generated TDLC groups. Suppose $A$ is hyperbolic relative to $\mathcal H_A=\{H_i^A\}_{i=1}^n$ and $B$ is hyperbolic relative to $\mathcal H_B=\{H_j^B\}_{j=1}^m$. Let $C_A\subset A$ and $C_B\subset B$ be two compact open subgroups. Now, form an amalgamated free product $A\ast_CB$ where $C$ is topologically isomorphic to $C_A$ and $C_B$. Let $X_A$ be a Cayley-Abels graph of $A$ whose vertices are the left cosets of $C_A$ in $A$. Similarly, let $X_B$ be the Cayley-Abels graph of $B$ whose vertices are the left cosets of $C_B$ in $B$. To construct a model space for a Cayley-Abels graph for $G$, we follow the construction of a Cayley-Abels graph of $G$ with respect to $\mathcal H_A\cup\mathcal H_B$ as in Subsection \ref{subsection-rel-hyp-amalgam-combination-construction}. In that construction, we replace the copies of the Cayley-Abels graph of $A$ and $B$ with respect to $\mathcal H_A$ and $\mathcal H_B$, by the copies of $X_A$ and $X_B$. Thus, we get a graph $X$ and it is quasiisometric to a Cayley-Abels graph of $G$. 

\vspace{.2cm}
{\bf Tree of augmented spaces:} Let $X_A^h$ and $X_B^h$ be the augmented Cayley-Abels graphs of $A$ and $B$ with, respectively. Now, replace $X_A$ and $X_B$ by $X^h_A$ and $X^h_B$ in $X$. Since each vertex space of $X$ is either isomorphic to $gX_A$ or isomorphic to $gX_B$ for some $g\in G$, replace each vertex space by its corresponding augmented Cayley-Abels graph. Thus, we get a new graph $X^h$ and a projection map $X^h\to T$. There is a bijection between the edges of $X$ connecting different Cayley-Abels graphs and the edges of $T$. We call them {\em lifts} of the edges of $T$. The preimage of each vertex $v\in V(T)$, denoted by $X_v^h$, is isomorphic to the augmented Cayley-Abels graph of the stabilizer $G_v$ of $v$ in $G$. These are called the {\em vertex spaces} of $X^h$. Since $A$ and $B$ are hyperbolic relative to $\mathcal{H}_A$ and $\mathcal{H}_B$, the vertex spaces of $X^h$ are uniform Gromov hyperbolic geodesic metric spaces, i.e. there exists a $\delta\geq 0$ such that each vertex space is $\delta$-hyperbolic. Hence, $X^h$ is a Gromov hyperbolic space \cite{BF}. Thus, $G$ is hyperbolic relative to $\mathcal H_A\cup \mathcal H_B.$ This gives a different proof of Theorem \ref{rel-combination-theorem-amalgam} for compactly generated relatively hyperbolic TDLC groups.
\vspace{.2cm}

{\bf Construction of the compactification $\overline{X^h}$.} 

{\em Boundaries of the stabilizers.} Let $\delta_{Stab}(X^h)$ be the set $G\times(\partial_{rel}A\cup\partial_{rel}B)$ divided by the equivalence relation induced by
	$$(g_1,\xi_1)\sim (g_2,\xi_2) \text{ if } \xi_1,\xi_2\in \partial_{rel}A, g_2^{-1}g_1\in A \text{ and } g_2^{-1}g_1\xi_1=\xi_2,$$
	$$(g_1,\xi_1)\sim (g_2,\xi_2) \text{ if } \xi_1,\xi_2\in \partial_{rel}B, g_2^{-1}g_1\in B \text{ and } g_2^{-1}g_1\xi_1=\xi_2.$$
 The set $\delta_{Stab}(X^h)$ comes with a natural action of $G$ on the left. This also comes with a natural projection onto the set of vertices of $T$. The preimage of each vertex $v\in T$ is denoted by $\partial_{rel}G_v$.
	
Let $\partial T$ denote the Gromov boundary of $T$. Then, we define the {\em boundary} of $X^h$ as
	$$\delta(X^h):=\delta_{Stab}(X^h)\sqcup\partial T.$$
	
Also, we define a set $\overline{X^h}$ (which will be called the {\em compactification} of $X^h$) as $$\overline{X^h}:=X^h\cup\delta(X^h).$$
This set has a natural action of $G$ and a natural map $p:\overline{X^h}\to \overline{T}$, where $\overline{T}=T\cup\partial T$. The preimage of a vertex $v\in V(T)$ is $X^h_v\cup\partial_{rel}G_v$ that is identified as a set with $\overline{X^h_v}.$
\subsection{HNN extension case.}\label{bowditch-boundary-construction-hnn}
Let $A$ be a compactly generated TDLC group. Suppose $A$ is hyperbolic relative to $\mathcal H$. Let $C$ and $C'$ be two topologically isomorphic compact open subgroups of $A$ and let $G=A{\ast_C}$ be an HNN extension. Let $X_A$ be the Cayley graph of $A$ such that the set of vertices is the left cosets of $C$ in $A$. Again, to construct a model space for a Cayley-Abels graph for $G$, we follow the construction of a Cayley-Abels graph of $G$ with respect to $\mathcal H$ as in Subsection \ref{subsection-rel-hyp-hnn-combination-construction}. In that construction, we replace the copies of the Cayley-Abels graph of $A$ with respect to $\mathcal H$ by the copies of $X_A$. Thus, we get a graph $X$ and it is quasiisometric to a Cayley-Abels graph of $G$. 
\vspace{.2cm}

{\bf Tree of Augmented spaces:} Let $X_A^h$ be the augmented Cayley-Abels graph of $A$. We now replace $X_A$ by $X_A^h$ in $X$. Since each vertex space of $X$ is isomorphic to $gX_A$ for some $g\in G$, replace each vertex space by its corresponding augmented Cayley-Abels graph. Thus, we get a new graph $X^h$ and a projection map $X^h\to T.$ The preimage of each vertex $v\in T$, denoted by $X_v^h$, is isomorphic to the augmented Cayley-Abels graph of the stabilizer $G_v$ of $v$ in $G$. These are called the {\em vertex spaces} of $X^h$. There is a bijection between the edges of $X$ connecting different Cayley-Abels
graphs and the edges of $T$. We call them {\em lifts} of the edges of $T$. Since $A$ is hyperbolic relative to $\mathcal H$, each vertex space of $X^h$ is uniformly a hyperbolic space. Hence, $X^h$ is a hyperbolic space \cite{BF}. Thus, $G$ is hyperbolic relative to $\mathcal H$. This also gives a different proof of Theorem \ref{theorem-rel-combination-theorem-hnn}.
\vspace{.2cm}

{\bf Construction of the compactification $\overline{X^h}.$} Let $\delta_{Stab}(X)$ be the set $G\times\partial_{rel} A$ divided by the equivalence relation induced by
 $$(g_1,\xi_1)\sim (g_2,\xi_2) \text{ if } \xi_1,\xi_2\in \partial A, \hspace{0.07cm}g_2^{-1}g_1\in A \text{ and } g_2^{-1}g_1\xi_1=\xi_2,$$
 
 The set $\delta_{Stab}(X^h)$ comes with a natural action of $G$ on the left. This also comes with a natural projection onto the set of vertices of $T$. The preimage of each vertex $v\in T$ is denoted by $\partial_{rel} G_v$.
 
 Let $\partial T$ denote the Gromov boundary of $T$. Then, as a set, we define the {\em boundary} of $X$ as $$\delta(X^h):=\delta_{Stab}(X^h)\sqcup\partial T.$$
 
Finally, as a set, we define the {\em compactification} of $X$ as $$\overline{X^h}:=X^h\cup\delta(X^h).$$
 This set comes with a natural action of $G$ and with a natural map $p:\overline{X^h}\to \overline{T}$, where $\overline{T}=T\cup\partial T$. The preimage of a vertex $v\in T$ is $X^h_v\cup\partial_{rel} G_v$ that is identified, as a set, with $\overline{X^h_v}$, which is the compactification of $X_v$. 
	
\subsection{Topology on $\overline{X^h}$.} Let $\overline{X^h}$ be either as in Subsection \ref{bowditch-boudary-construction-amalgam} or as in Subsection \ref{bowditch-boundary-construction-hnn}. For a point $x\in X^h$, we set a basis of open neighborhoods of $x$ in $X^h$ to be a basis of open neighborhoods of $x$ in $\overline{X^h}$. Now, we define a basis of open neighborhoods for points of $\delta(X^h)$. Fix a vertex $v_0$ of $T$.
	
(1) Let $\xi\in\delta_{Stab}(X^h)$. Suppose $v$ is the vertex of $T$ such that $\xi\in\partial_{rel}G_v$. Let $U$ be an open neighborhood of $\xi$ in $\overline{X^h_v}$. Define $V_U$ to be the set of all $z\in\overline{X^h}$ such that $p(z)\neq v$, and the first edge of the geodesic in $T$ from $v$ to $p(z)$ lifts to an edge of $X^h$ that is glued to a point of $U$. Then we set $$V_U(\xi):=U\cup V_U.$$ A neighborhood basis of $\xi$ in $\overline{X^h}$ is a collection of set $V_U(\xi)$ where $U$ runs over some neighborhood basis of $\xi$ in $\overline{X^h_v}$.
	
(2) Let $\eta\in \partial T$. Let $T_n(\eta)$ be the subtree of $T$ consisting of those elements $x$ of $T$ for which the first $n$ edges of $[v_0,x]$ and $[v_0,\eta)$ are the same. Let $\partial T_n{(\eta)}$ denote the Gromov boundary of $T_n{(\eta)}$, and let $\overline{T_n(\eta)}=T_n(\eta)\cup\partial T_n(\eta)$. We define $$V_n(\eta)=p^{-1}(\overline{T_n(\eta)}).$$ We take the collection $\{V_n(\eta):n\geq 1\}$ as a basis of open neighborhoods of $\eta$ in $\overline{X^h}$.
	
We skip a verification that the above collections of sets satisfy the axioms for the basis of open neighborhoods. 
The rest of the subsection is devoted to showing the following.
\begin{proposition}
		The space $\delta(X^h)$ is compact metrizable and there is a $G$-equivariant homeomorphism from $\partial_{rel}(G)$ to $\delta(X^h)$.
	\end{proposition} 
{\em A word on the proof:} From the construction of $X^h$, it follows that there is a $G$-equivariant quasiisometry from the augmented Cayley-Abels graph of $G$ to $X^h$. Hence, $\partial_{rel} G$ is $G$-equivariantly homeomorphic to the Gromov boundary $\partial X^h$ of $X^h$. From the description of geodesic rays in $X^h$ as in \cite[p.7]{martin-swiat}, one can define a natural bijection from $\partial_{rel}G$ to $\delta(X^h)$. As $\partial_{rel}G$ is compact and $\delta(X^h)$ is Hausdorff, to show that this bijection is a homeomorphism, it is sufficient to show that it is continuous. This follows from \cite[Section 3, pp.280-281]{martin-swiat}.
\subsection{Homeomorphism types of the Bowditch boundaries of $A\ast_C B$ and $A\ast_C$}\label{subsection-homeo-type-rel-hyp-amalgam-hnn}

We start here by recalling a dense amalgam construction of compact metric spaces from \cite{swiatkowski-dense-amalgam}. Given a collection of non-empty compact metric spaces $X_1,\dots,X_k$, Jacek \`{S}wi\polhk{a}tkowski \cite{swiatkowski-dense-amalgam} introduced the notion of {\em dense amalgam} of $X_i$'s. To prove our main result of this section, we use the following characterization of dense amalgam.

\begin{theorem}\label{dense-amalgam-char}
    \cite[Theorem 0.2]{swiatkowski-dense-amalgam} For $k\geq 1$, let $X_1,\dots,X_k$ are non-empty compact metric spaces. Let $Y$ be a compact metric space equipped with a countable infinite family $\mathcal Y$ of subsets, partitioned as $$\mathcal Y=\mathcal Y_1 \sqcup\dots\sqcup \mathcal Y_k$$ such that the following hold:
    \begin{enumerate}
        \item The subsets in $\mathcal Y$ are pairwise disjoint and for each $i\in\{1,\dots,k\}$ the family $\mathcal Y_i$ consists of embedded copies of the space $X_i.$
        \item The family $\mathcal Y$ is {\em null},i.e. for any metric on $Y$ compatible with the topology the diameters of sets in $\mathcal Y$ converge to $0$.
        \item Each $Z\in\mathcal Y$ is a boundary subset of $Y$ (i.e. its complement is dense).
        \item For each $i$, the union of the family $\mathcal Y_i$ is dense in $Y$.
        \item Any two points of $Y$ which do not belong to the same subset of $\mathcal Y$ can be separated from each other by an open and closed subset $H\subset Y$ which is
$\mathcal Y$-saturated (i.e. such that any element of $\mathcal Y$ is either contained in or disjoint
with $H$).
\end{enumerate}
Then, $Y$ is homeomorphic to the dense amalgam $\widetilde{\sqcup}(X_1,\dots,X_k)$ of $X_i$'s.
\end{theorem}

\begin{definition}\cite{swiatkowski-dense-amalgam} Let $X_1,\dots,X_k$ be a collection of
nonempty metric spaces, for some $k\geq 1$. A compact metric space $Y$ is {\em $(X_1,\dots, X_k)$-regular} if it can be equipped with a family $\mathcal Y$ of subspaces satisfying
conditions $(1)-(5)$ from Theorem \ref{dense-amalgam-char}.
\end{definition}

From Theorem \ref{dense-amalgam-char}, it follows that if $Y$ is $(X_1,\dots,X_2)$-regular then $Y$ is homeomorphic to the dense amalgam $\widetilde{\sqcup}(X_1,\dots,X_k).$ We now adapt the dense amalgam construction in our setup.

Let $A\ast_C B$ be an amalgamated free product as in Subsection \ref{bowditch-boudary-construction-amalgam}. Let $\partial_{rel} A$ be the Bowditch boundary of $A$ and let $\partial_{rel} B$ be the Bowditch boundary of $B$. Let $\delta(X^h)$ be the Bowditch boundary of $A\ast_C B$ as constructed in Subsection\ref{bowditch-boudary-construction-amalgam}. We aim to show that $\delta(X^h)$ is homeomorphic to the dense amalgam $\widetilde{\sqcup}(\partial_{rel} A,\partial_{rel} B)$ of $\partial_{rel} A$ and $\partial_{rel} B$. To prove that it is sufficient to prove that $\delta(X^h)$ is $(\partial_{rel} A,\partial_{rel} B)$-regular. Towards this direction, we now describe a family $\mathcal Y$ of subspaces of $\delta(X^h)$. Define $$\mathcal Y_A=\{\partial_{rel} G_v: \pi(v)=v_A\},$$ $$\mathcal Y_B=\{\partial_{rel} G_v:\pi(v)=v_B\}.$$ 

\begin{proposition}\label{main-prop}
    The family $\mathcal Y$ of subspaces of $\delta(X^h)$ satisfies the conditions (1)-(5) of Theorem \ref{dense-amalgam-char}.
\end{proposition}

\begin{proof}
    From the definition of the topology on $\delta(X^h)$, it follows that the topology on $\partial_{rel} G_v$ agrees with the subspace topology induced from $\delta(X^h)$ for all $v\in V(T).$ Thus, for all $v\in V(T)$, $\partial_{rel} G_v$ is embedded in $\delta(X^h).$ Clearly, elements in $\mathcal Y_A\cup \mathcal Y_B$ are disjoint. By Lemma \ref{lemma-countable-vertices-cayley-abels-graph}, the family $\mathcal Y$ is countable. This verifies condition (1). 

    To check condition (2), i.e. nullness of the family $\mathcal Y$, we need to show that for each open covering $\mathcal U$ of $\delta(X^h)$ there is a finite subfamily $\mathcal A \subset \mathcal Y$ such that for every $Z\in \mathcal Y \setminus \{\mathcal A\}$ there is $U\in \mathcal U$ that contains $Z$. Since $\delta(X^h)$ is compact, without loss of generality, we may assume that $\mathcal U$ is finite and consists of sets from bases of open neighborhoods of points. Let $$\mathcal U=\{V_{U_1}(\xi_1),\dots,V_{U_l}(\xi_l),V_{n_1}(\eta_1),\dots,V_{n_m}(\eta_m)\}.$$ We prove that the above property holds for $\mathcal U$ with $$\mathcal A=\{\partial_{rel} G_{p(\xi_1)},\dots,\partial_{rel} G_{p(\xi_l)}\}.$$ Let $Z\in\mathcal Y\setminus \{\mathcal A\}.$ Since $\mathcal U$ is a covering of $\delta(X^h)$, either $Z\cap V_{U_i}(\xi_i)\neq\phi$ or $Z\cap V_{n_j}(\eta_j)\neq \phi$ for some $1\leq i\leq l$ and $1\leq j\leq m$. Suppose $Z\cap V_{U_i}(\xi_i)\neq \phi.$ Since $Z\notin\mathcal A$, it follows from the definition of $V_{U_i}(\xi_i)$ that $Z\subset V_{U_i}(\xi_i).$ Suppose $Z\cap V_{n_j}(\eta_j)\neq\phi$. Then, again by definition of $V_{n_j}(\eta_j)$, $Z\subset V_{n_j}(\eta_j).$ This completes the verification of condition (2). 

    For condition (3), let $Z\in\mathcal Y_A$. Suppose $Z=\partial_{rel} A$. A similar computation can be done when $Z=\partial_{rel} G_v$ and $\pi(v)=v_A.$ We need to show that the complement of $Z$ in $\delta(X^h)$ is dense in $\delta(X^h)$. Let $\xi\in\partial_{rel} A$. Since $X^h_A$ is dense in $\overline{X^h_A}$ and vertices of $X^h_A$ are the cosets of $C$ in $A$, there exists a sequence of cosets $\{a_n C\}$  that converges to $\xi$. Note that the edges of $T$ adjacent to $v_A$ are nothing but the cosets of $C$ in $A$. Let $U$ be a neighborhood of $\xi$ in $\partial_{rel} A$. Thus, the neighborhood $V_U(\xi)$ of $\xi$ in $\delta(X^h)$ contains all those $\beta\in \delta(X^h)\setminus\{\partial_{rel} A\}$ such that the coset corresponding to the first edge from $v_A$ to $p(\beta)$ is some $a_nC$ and it belongs to $U$ (as an element of $X_A$). Since $\{a_nC\}$ converges to $\xi$, $V_U(\xi)\cap (\delta(X^h)\setminus\{\partial_{rel} A\})\neq\phi.$ A similar analysis can be done for $Z\in\mathcal Y_B.$ This completes the verification of condition (3).

For condition (4), we only verify that $\mathcal Y_A$ is dense in $\delta (X^h).$ A similar computation can be done for $\mathcal Y_B.$ Let $\xi\in\partial_{rel} G_v$ where $p(v)=v_B.$ We show that any neighborhood of $\xi$ in $\delta(X^h)$ intersects some element of $\mathcal Y_A$. By the discussion in condition (3), any neighborhood $U$ of $\xi$ in $\overline{X^h_v}$ contains at least one vertex of an edge $\Tilde{e}$ of $X^h$ which is not in any vertex space of $X^h$. Suppose $\Tilde{e}$ is the lift of the edge $e$ in $T$. Let $w$ be a vertex on the geodesic segment of $T$ starting from $v$ whose first edge is $e$ and $\pi(w)=v_A$. Then, from the definition of $V_U(\xi)$, it follows that $\partial_{rel} G_w\subset V_U(\xi)$. Observe that any neighborhood around a boundary point of $T$ contains copies of $\partial_{rel} A$ as well as copies of $\partial_{rel} B$.  This completes the verification of condition (4).

For condition (5), first of all, observe that a neighborhood $V_n(\eta)$ around a boundary point $\eta$ of $T$ is closed as well as open in $\delta (X^h)$. Also, $V_n(\eta)$ is $\mathcal Y$-saturated. Let $\xi_1,\xi_2$ be points of $\delta(X^h)$ which do not lie in the same subset of $\mathcal Y$. Let $\xi_1\in\partial_{rel} G_v$ and $\xi_2\in\partial_{rel} G_w$ and $v\neq w.$ 
Choose $\eta\in\partial T$ such that the geodesic ray $[v_0,\eta)$ does not pass through $v$ but $w\in [v_0,\eta).$ Let $n$ be the length of the geodesic segment $[v_0,\eta)\cap[v_0,w]$. Then, from definition of $V_{n-1}(\eta)$, we see that $V_{n-1}(\eta)$ contains $\xi_2$ but does not contain $\xi_1.$ This completes the verification of condition (5). 
\end{proof}
Now, we are ready to give a proof of Theorem \ref{theorem-rel-hyp-homeo-type-amal-case}. 
\vspace{.2cm}

Proof of Theorem \ref{theorem-rel-hyp-homeo-type-amal-case}: Let $T$ and $T'$ be the Bass-Serre trees of $A\ast_C B$ and $A'\ast_C B'$ respectively. Let $X^h$ and $X'^h$ be the augmented Cayley-Abels graphs of $G$ and $G'$, respectively, as constructed in Subsection \ref{bowditch-boudary-construction-amalgam}.
 We need to consider the following three cases:

{\bf Case 1.} Firstly, observe that if the Bowditch boundary of a relatively hyperbolic TDLC group $G$ is empty, then $G$ is compact. If $\partial_{rel} A$ and $\partial_{rel} B$ are both empty, then $\partial_{rel} A'$ and $\partial_{rel} B'$ are also empty. Thus, $A, B$ and $A', B'$ are all compact groups. This implies that $T$ and $T'$ are uniformly locally finite, and the Cayley-Abels graphs of $G$ and $G'$ are quasiisometric to $T$ and $T'$, respectively. Hence, $\partial_{rel} G$ and $\partial_{rel} G'$ are both homeomorphic to the Cantor set.

{\bf Case 2.} Suppose $\partial_{rel} A\neq \phi$ and $\partial_{rel} B=\phi.$ Then, $\partial_{rel} A'\neq \phi$ and $\partial_{rel} B'=\phi$. By Proposition \ref{main-prop}, $\partial_{rel} G$ and $\partial_{rel} G'$ are homeomorphic to $\widetilde{\sqcup}(\partial_{rel} A)$ and $\widetilde{\sqcup}(\partial_{rel} A')$. Since $\partial_{rel} A$ is homeomorphic to $\partial_{rel} A'$, by \cite[Proposition 3.1]{swiatkowski-dense-amalgam}, $\widetilde{\sqcup}(\partial_{rel} A)$ is homeomorphic to $\widetilde{\sqcup}(\partial_{rel} A')$. 

{\bf Case 3.} Suppose $\partial_{rel} A\neq \phi$ and $\partial_{rel} B\neq \phi$. By Proposition \ref{main-prop}, $\partial_{rel} G$ and $\partial_{rel} G'$ are homeomorphic to $\widetilde{\sqcup}(\partial_{rel} A,\partial_{rel} B)$ and $\widetilde{\sqcup}(\partial_{rel} A',\partial_{rel} B')$. Now, the proof follows from \cite[Proposition 3.1]{swiatkowski-dense-amalgam}.      \qed

\begin{remark}\label{remark-homeo-amalgam}
    Suppose $G$ and $G'$ are as in Theorem \ref{theorem-rel-hyp-homeo-type-amal-case}. If the index of $C$ in $A$ and $B$ are equal to $2$, and the index of $C'$ in $A'$ or $B'$ is at least $3$, then $\partial_{rel} G$ is not homeomorphic to $\partial_{rel} G'$ since $G$ has two rough ends and $G'$ has infinitely many rough ends.
\end{remark}
Given a topological group $A$ and an open subgroup $C$ of $A$, the group $A\ast_C A$ denotes the double of $A$ along $C$.

\begin{corollary}
    Suppose $G=A\ast_C B$ where $A$ is non-compact and hyperbolic relative to $\mathcal H_A$, and $B$ is compact. Then, the Bowditch boundary of $G$ with respect to $\mathcal H_A$ is homeomorphic to the Bowditch boundary of $H=A\ast_C A$ with respect to $\mathcal H_A\sqcup\mathcal H_A$ (By Theorem \ref{rel-combination-theorem-amalgam}, $G$ and $H$ are hyperbolic relative to $\mathcal H_A$ and $\mathcal H_A\sqcup\mathcal H_A$, respectively).
\end{corollary}
\begin{proof}
    From Case 2. of the proof of Theorem \ref{theorem-rel-hyp-homeo-type-amal-case}, we see that $\partial_{rel} G$ is homoeomorphic to $\widetilde{\sqcup}(\partial_{rel} A)$. However, by \cite[0.1 Proposition]{swiatkowski-dense-amalgam}, $\widetilde{\sqcup}(\partial_{rel} A)$ is homeomorphic to $\widetilde{\sqcup}(\partial_{rel} A,\partial_{rel} A)$. But, $\partial_{rel} H$ is also homeomorphic to $\widetilde{\sqcup}(\partial_{rel} A,\partial_{rel} A)$. This completes the proof.
\end{proof}
Now, we discuss the homomorphism types of the Bowditch boundaries of HNN extensions of relatively hyperbolic groups and a proof of Theorem \ref{theorem-rel-hyp-homeo-type-hnn}.
\vspace{.2cm}

Proof of Theorem \ref{theorem-rel-hyp-homeo-type-hnn}: Again, to prove the theorem, we use the dense amalgam technique of compact metric spaces. Let $\partial_{rel} A$ and $\partial_{rel} A'$ denote the Bowditch boundaries of $A$ and $A'$ with respect to $\mathcal H_A$ and $\mathcal H_A'$, respectively. There are two cases to be considered:
    
    {\bf Case 1.} Suppose $\partial_{rel} A=\phi$. Then, $\partial_{rel} A'=\phi$, and hence $A$ and $A'$ are both compact groups. Then, Cayley-Abels graphs of $G$ and $G'$ are both quasiisometric to Bass-Serre trees of $G$ and $G'$, respectively. Since, in this case, Bass-Serre trees are uniformly locally finite, hence $\partial_{rel} G$ and $\partial_{rel} G'$ are homeomorphic to the Cantor sets.
    
    {\bf Case 2.} Suppose $\partial_{rel} A\neq \phi$. Then, $\partial_{rel} A'$ is also non-empty. Using \cite[Proposition 3.1]{swiatkowski-dense-amalgam}, it is sufficient to prove that $\partial_{rel} G$ is homeomorphic to $\widetilde{\sqcup}(\partial_{rel} A)$ and $\partial_{rel} G'$ is homeomorphic to $\widetilde{\sqcup}(\partial_{rel} A').$  The rest of the proof is devoted to proving this. Let $\delta(X^h)$ and $\delta(X'^h)$ be the Bowditch boundaries of $G$ and $G'$, respectively, as constructed in Subsection \ref{bowditch-boundary-construction-hnn}. We show that $\delta(X^h)$ and $\delta(X'^h)$ are $(\partial_{rel} A)$-regular and $(\partial_{rel} A')$-regular, respectively. We define a family of subspaces $\mathcal Y_A$ and $\mathcal Y_{A'}$ of $\delta(X^h)$ and $\delta(X'^h)$. $$\mathcal Y_A=\delta_{Stab}(X^h)$$ $$\mathcal Y_{A'}=\delta_{Stab}(X'^h).$$ As in the proof of Proposition \ref{main-prop}, $\mathcal Y_A$ and $\mathcal Y_{A'}$ satisfy the conditions (1)-(5) of Theorem \ref{dense-amalgam-char}. Hence, we are done. \qed
    
Let $(\mathcal G,\mathcal Z)$ be a graph of relatively hyperbolic TDLC groups with compact edge groups over a finite graph $\mathcal Z$. We denote by $h(\mathcal G,\mathcal Z)$ the homeomorphism types (without multiplicity) of non-empty Bowditch boundaries of the vertex groups of $(\mathcal G,\mathcal Z).$ Using induction and combining Theorem \ref{theorem-rel-hyp-homeo-type-amal-case} and Theorem \ref{theorem-rel-hyp-homeo-type-hnn}, we have the following:

\begin{theorem}\label{theorem-rel-hyp-homeo-type-general}
Let $(\mathcal G,\mathcal Z_1)$ and $(\mathcal G,\mathcal Z_2)$ be two graphs of groups over two finite graphs $\mathcal Z_1$ and $\mathcal Z_2$ with fundamental group $G_1$ and $G_2$ such that the following hold:
\begin{enumerate}
    \item[{$(i)$}] Each vertex group of $(\mathcal G,\mathcal Z_1)$ and $(\mathcal G,\mathcal Z_2)$ is relatively hyperbolic.
    \item[{$(ii)$}] Each edge group of $(\mathcal G,\mathcal Z_1)$ and $(\mathcal G,\mathcal Z_2)$ is compact in the adjacent vertex groups.
    \item [{$(iii)$}] The rough ends of $G$ and $G'$ are infinite.
\end{enumerate}
    If $h(\mathcal G,\mathcal Z_1)=h(\mathcal G,\mathcal Z_2)$, then the Bowditch boundary of $G_1$ is homeomorphic to the Bowditch boundary of $G_2.$
\end{theorem}

Since every hyperbolic TDLC group is relatively hyperbolic with respect to the empty collection of subgroups, we obtain the following:
\begin{corollary}\label{corollary-hyp-homeo-type-amalgam}
    Suppose $G=A\ast_C B$ and $G'=A'\ast_{C'} B'$ are amalgamated free products of compactly generated hyperbolic TDLC groups such that $C$ and $C'$ are compact, and the rough ends of $G$ and $G'$ are infinite. Suppose the Gromov boundary of $A$ and $B$ is homeomorphic to the Gromov boundary of $A'$ and $B'$, respectively.
    Then, the Gromov boundary of $G$ is homeomorphic to the Gromov boundary of $G'$.    \qed
    \end{corollary}

    \begin{corollary}\label{corollary-hyp-homeo-type-hnn}
    Suppose $G=A\ast_C$ and $G'=A'\ast_{C'}$ are HNN extensions of compactly generated hyperbolic TDLC groups such that $C$ and $C'$ are compact, and the rough ends of $G$ and $G'$ are infinite. Suppose the Gromov boundary of $A$ is homeomorphic to the Gromov boundary of $A'$.
    Then, the Gromov boundary of $G$ is homeomorphic to the Gromov boundary of $G'$.    \qed
    \end{corollary}

 Using induction and combining the previous two corollaries, one can have a similar statement as Theorem \ref{theorem-rel-hyp-homeo-type-general}. 
    
\section{Connectedness of the Bowditch boundary}\label{section-connected-bowditch-boundary}
This section is devoted to proving Theorem \ref{theorem-connectedness-bowditch-boundary}, which is well known for discrete relatively hyperbolic groups (\cite[Proposition 10.1]{bowditch-relhyp}). In this section, all relatively hyperbolic TDLC groups are in the sense of TDRH-II.

Let $G$ and $\mathcal H$ be as in Theorem \ref{theorem-connectedness-bowditch-boundary}. Before we start the proof, we need a little preparation. Let $X$ and $X^h$ be a Cayley-Abels graph, and an augmented Cayley-Abels graph of $G$, respectively, as in Subsection \ref{subsection-tdrh-II}. Let $\partial_{rel}G$ be the Bowditch boundary of $G$ with respect to $\mathcal H.$ Note that $\overline{X^h}:=X^h\cup\partial_{rel}G$ is a compact metrizable space, and $X^h$ is dense in $\overline{X^h}$. Define $$\overline{G}:=V(X)\cup \partial_{rel}G.$$ We denote the induced metric on $\overline{G}$ by $d_{\Delta}$. Define a map $\pi:V(X)\to \partial_{rel}G$ as 
\begin{equation}
	\pi(v) = \xi, \text{ where } \xi \in \partial_{rel}G \text{ such that } d_{\triangle}(v, \xi) = d_{\Delta}(v,\partial_{rel}G)
\end{equation}

Note that (a) such a $\xi$ exists as $\partial_{rel}G$ is compact, (b) $\xi$ may not be unique; however, uniqueness is not required for our purposes.

\begin{lemma}\label{lemma-dense-lemma}
    If $\{v_i\}\subset V(X)$ is a sequence converging to $\xi\in\partial_{rel}G$ with respect to $d_{\Delta}$, then $\{\pi(v_i)\}$ also converges to $\xi$. 
	
\end{lemma}
\begin{proof}
    From the definition of $\pi$, we have $d_{\triangle}(v_i,\pi(v_i))\to 0$ as $i\to\infty$. An application of the triangle inequality then implies that $\pi(v_i)\to \xi$ as $i\to\infty$.
\end{proof}
Now, we are ready to prove Theorem \ref{theorem-connectedness-bowditch-boundary}.
\vspace{.2cm}

Proof of Theorem \ref{theorem-connectedness-bowditch-boundary}: Let $\pi$ be the map as defined in Equation (1). Let if possible $$\partial_{rel}G=U_1\sqcup U_2$$ where $U_1$ and $U_2$ are non-empty disjoint open subsets of $\partial_{rel}G$. Let $D_i=\pi^{-1}(U_i)$ for $i=1,2$. Denote the closure of $D_i$ in $\overline{G}$ by cl$(D_i)$ (with respect to the topology induced by $d_{\Delta}$).

{\bf Claim:} $D_1$ and $D_2$ are non-empty, disjoint, and satisfy cl$(D_i)=D_i\cup U_i$ for $i=1,2$.

By Lemma \ref{lemma-dense-lemma}, it is clear that cl$(D_i)\subseteq D_i \cup U_i$. Conversely, let $\xi \in U_1$. Since $V(X)$ is dense in $\overline{G}$, let $\{v_n\}$ be a sequence in $V(X)$ such that $v_n \to \xi$ (in the original topology of $\overline{G}$). As the topology induced by $d_{\Delta}$ on $\partial_{rel}G$ is same as the original topology on $\partial_{rel}G$, $\{v_n\}\to \xi$ with respect to the metric $d_{\Delta}$. Suppose there exists a subsequence of $\{v_{n}\}$ contained in $D_2$. Then, $\xi \in$ cl$(D_2)$, which in turn implies that $\xi \in U_2$. This gives a contradiction as $U_1$ and $U_2$ are disjoint. Hence, $\{v_n\}$ is eventually contained in $D_1$ and therefore, $\xi\in$ cl$(D_1)$. Similarly, one can show that $U_2\subseteq$ cl$(D_2)$. This also shows that $D_1$ and $D_2$ are non-empty, disjoint subsets of $V(X)$. Hence, the claim.

Since $U_1$ and $U_2$ are non-empty, let $\xi_1 \in U_1$ and $\xi_2 \in U_2$. As in the proof of the converse of the claim, there exist sequences $\{v_n\}$ and $\{w_n\}$ in $D_1$ and $D_2$, respectively, such that $v_n \to \xi_1$ and $w_n \to \xi_2$. Let $d$ denote the Cayley-Abels graph metric on $X$. Since $d$-balls in $X$ have only finitely many vertices of $V(X)$, up to passing to subsequences, we can assume that $\{v_n\}$ and $\{w_n\}$ are unbounded in $X$. For every $m\in \mathbb{N}$, denote $K_m$ the closed $d$-metric ball of radius $m$ about the identity coset (coset of compact open subgroup whose representative is identity) which is an element of $V(X)$. Since $G$ has one rough end, there exist subsequences $\{v_{n_m}\}$ and $\{w_{n_m}\}$ and a sequence of paths $\{\gamma_{m}\}$ joining $v_{n_m}$ and $w_{n_m}$ such that $\gamma_m$ is contained in $X\setminus K_m$. Let $e_m=[a_m,b_m]$ be an edge in $\gamma_{m}$ such that $a_m\in D_1$ and $b_m\in D_2$ (such an edge always exists). Note that the sequence $\{a_m\}$ is unbounded in $X$. Since $\overline{G}$ is compact, there exists a subsequence $\{a_{m_k}\}$ that converges to $\eta \in \partial_{rel}G$. Since $d(a_{m_k},b_{m_k})=1$, it follows $b_{m_k} \to \eta$ in $\overline{G}$. Note that $\{a_{m_k}\}\to \eta$ too in the topology induced by the metric $d_{\Delta}$ ($d_{\Delta}$ induces the original topology on $\partial_{rel}G$. Thus, $\eta\in U_1$. By the same logic, $\eta \in U_2$. This gives us a contradiction as $U_1$ and $U_2$ are disjoint. Hence, we have the desired result. 
\qed
\section{Applications}\label{section-applications}
This section is devoted to proving partial converses of Theorem \ref{theorem-rel-hyp-homeo-type-amal-case} and Theorem \ref{theorem-rel-hyp-homeo-type-hnn}. For that, first of all, we analyze the connected component of the Bowditch boundary of a relatively hyperbolic amalgamated free product or HNN extension of relatively hyperbolic TDLC groups.

\begin{proposition}\label{proposition-connected-component}
Let $G$ be as in Theorem \ref{theorem-rel-hyp-homeo-type-amal-case} or Theorem \ref{theorem-rel-hyp-homeo-type-hnn} with Bass-Serre tree $T$. Additionally, assume that $A$ and $B$ have one rough end. Then, we have the following:
\begin{enumerate}
    \item Each $\eta\in\partial T$ gives a connected component of $\partial G$.
    \item For each $v\in V(T)$, the Bowditch boundary $\partial G_v$ is a connected component of $\partial G$.
\end{enumerate}
\end{proposition}
By Theorem \ref{theorem-connectedness-bowditch-boundary}, $\partial_{rel} G$ is a connected topological space. Thus, the proof of this proposition follows from the proof of \cite[Propositions 6.3 and 6.4]{martin-swiat}. Hence, we skip its proof. For an amalgamated free product $A\ast_C B$ or an HNN extension $A\ast_C$ of relatively hyperbolic TDLC groups, $h(A\ast_C B)$ and $h(A\ast_C)$ denote the set of homeomorphism types of the non-empty Bowditch boundary of the vertex groups. We are now ready to prove the following.
\begin{theorem}\label{theorem-converse-homeo-type-rel-amal}
Suppose $G=A\ast_C B$ and $G'=A'\ast_{C'} B'$ are as in Theorem \ref{theorem-rel-hyp-homeo-type-amal-case}. Additionally, assume that $A, B, A'$, and $B'$ have one rough end. If $\partial_{rel}G$ is homeomorphic to $\partial_{rel}G'$ then $h(A\ast_C B)=h(A'\ast_{C'}B').$    
\end{theorem}
\begin{proof}
    Suppose $\partial_{rel} G$ is homeomorphic to $\partial_{rel} G'$. Let $T$ and $T'$ be the Bass-Serre trees of $G$ and $G'$. By Proposition \ref{proposition-connected-component}, for each $v\in V(T)$ and $v'\in V(T')$, $\partial G_v$ and $\partial G_{v'}$ are connected components in $\partial G$ and $\partial G'$, respectively. Since a homeomorphism preserves connected components, the theorem follows.
\end{proof}
A similar proof also holds for the HNN extension of relatively hyperbolic TDLC groups. Thus, we omit the details, but we include the statement for the sake of completeness.
\begin{theorem}\label{converse-homeo-type-rel-hnn}
Suppose $G=A\ast_C$ and $G=A'\ast_{C'}$ are as in Theorem \ref{theorem-rel-hyp-homeo-type-hnn}. Additionally, assume that $A$ and $A'$ have one rough end. If $\partial_{rel}G$ is homeomorhic to $\partial_{rel} G'$, then $h(A\ast_C)=h(A'\ast_{C'}).$ 
\qed
\end{theorem}
Since a hyperbolic TDLC group is hyperbolic relative to an empty collection of subgroups, we also obtain partial converses of Corollaries \ref{corollary-hyp-homeo-type-amalgam} and \ref{corollary-hyp-homeo-type-hnn}.

\begin{theorem}\label{theroem-converse-homeo-type-hyp}
$(1)$ Suppose $G=A\ast_C B$ and $G'=A'\ast_C B'$ are as in Corollary \ref{corollary-hyp-homeo-type-amalgam}. Additionally, assume that $A,B$, and $A',B'$ have one rough end. If $\partial G$ is homeomorphic to $\partial G'$, then $h(A\ast_C B)=h(A'\ast_{C'} B')$.

$(2)$ Suppose $G=A\ast_C$ and $G'=A'\ast_{C'}$ are as in Corollary \ref{corollary-hyp-homeo-type-hnn}. Additionally, assume that $A$ and $A'$ have one rough end. If $\partial G$ is homeomorphic to $\partial G'$, then $h(A\ast_C)=h(A'\ast_{C'})$.
\qed
\end{theorem}

Using Stallings' theorem on ends of discrete groups and accessibility of finitely presented groups, it is well known that if the Gromov boundary of a hyperbolic group $G$ is totally disconnected, then $G$ is virtually free, i.e. $G$ admits a finite graph of groups decomposition such that vertex groups are finite. We now aim to prove an analog of this for hyperbolic TDLC groups. For that, we recall the definition of accessibility for compactly generated TDLC groups. 
    
\begin{definition}\label{definition-accessible-group}
    A compactly generated TDLC group $G$ is said to be {\em accessible} if it has an action on a tree $T$ such that:
    \begin{enumerate}
        \item The number of orbits of $G$ on the edge of $T$ is finite.
        \item The stabilizers of edges in $T$ are compact open subgroups of $G$.
        \item Every stabilizer of a vertex in $T$ is a compactly generated open subgroup of $G$ and has at most one rough end.
    \end{enumerate}
\end{definition}
In \cite[Corollary 19.46]{cornulier-qi-classification}, the author proved that hyperbolic TDLC groups are accessible.

\vspace{.2cm}
Proof of Proposition \ref{proposition-total-disconnection-implies-splitting}:
        Since $G$ is a hyperbolic TDLC group, it is accessible. Thus, by Definition \ref{definition-accessible-group}, $G$ has a finite graph of groups $(\mathcal G,\mathcal Z)$ decomposition such that the edge groups are compact open in $G$, and the vertex groups are compactly generated TDLC. Note that, by Corollaries \ref{corollary-combination-amalgam-hyp-tdlc} and \ref{corollary-combination-hnn-hyperbolic-tdlc}, we see that the vertex groups of $(\mathcal G,\mathcal Z)$ are hyperbolic. Since $\partial G$ is totally disconnected, and the Gromov boundaries of vertex groups having one rough end are connected components in $\partial G$, each vertex group is compact. Hence, we are done.
        \qed
        
We end the paper with an example that demonstrates the situation in Proposition \ref{proposition-total-disconnection-implies-splitting}.

\begin{example}\label{example-aut(t_d)}
      For $d\geq 3$, let $T_d$ be a regular tree. Then, it is known that the automorphism group ${\rm Aut}(T_d)$ with compact open topology is a compactly generated TDLC group. It turns out that $T_d$ itself is a Cayley-Abels graph for ${\rm Aut}(T_d)$. Therefore, it is a hyperbolic TDLC group and the Gromov boundary of ${\rm Aut}(T_d)$ is homeomorphic to the Cantor set. Thus, by Proposition \ref{proposition-total-disconnection-implies-splitting}, ${\rm Aut}(T_d)$ admits a finite graph of groups decomposition such that the vertex groups are compact. However, this can also be seen by using Bass-Serre theory. Since $T_d$ is a Cayley-Abels graph of ${\rm Aut}(T_d)$, the vertex stabilizers are compact open and the action is transitive. So ${\rm Aut}(T_d)$ is a multiple HNN extension of the stabilizer of $T_d$. Also, by Theorem \ref{corollary-combination-amalgam-hyp-tdlc}, $G={\rm Aut}(T_d)\ast_C{\rm Aut}(T_d)$, where $C=Stab_{Aut(T_d)}(v)$ for some $v\in V(T_d)$, is a hyperbolic TDLC group. By our construction of the Gromov boundary of $G$, we see that $\partial G$ is homeomorphic to the Cantor set.

  \end{example} 

\vspace{.2cm}
\noindent
{\bf Acknowledgements.} The authors thank Eduardo Mart\`{i}nez-Pedroja for his helpful comments on an earlier draft of this paper, and for bringing our attention to his work with Hadi Bigdely.

\vspace{.2cm}
  \noindent
{\bf Conflict of interest:}
On behalf of all authors, the corresponding author states that there is no conflict of interest.

\bibliography{rp}
\bibliographystyle{amsalpha}

\end{document}